\documentclass[oneside,english,12pt]{amsart}
\usepackage{}
\usepackage{txfonts}
\usepackage{amsfonts}
\usepackage{lmodern}

\usepackage[T1]{fontenc}
\usepackage[latin9]{inputenc}
\usepackage{geometry}
\usepackage{amsthm}
\usepackage{amssymb}
\usepackage{amssymb,amsmath,amsthm,tikz}
\usetikzlibrary{matrix,arrows,decorations.pathmorphing}

\usepackage[bookmarks=true, bookmarksopen=true,%
bookmarksdepth=3,bookmarksopenlevel=2,%
colorlinks=true,%
linkcolor=blue,%
citecolor=blue,%
filecolor=blue,%
menucolor=blue,%
urlcolor=blue]{hyperref}

\title{Augmentations and link group representations}

\author{Honghao Gao}
\address{Department of Mathematics, Michigan State University, 619 Red Cedar Road, East Lansing, MI 48824, USA}
\email{gaohongh@msu.edu}

\makeatletter

\numberwithin{equation}{section}
\numberwithin{figure}{section}

\theoremstyle{plain}
\newtheorem{thm}{Theorem}[section]
\newtheorem{lem}[thm]{Lemma}

\newtheorem{prop}[thm]{Proposition}

\newtheorem{thm-defn}[thm]{Theorem-Definition}

\newtheorem{prop-def}[thm]{Proposition-Definition}

\theoremstyle{definition}
\newtheorem{defn}[thm]{Definition}
\newtheorem{eg}[thm]{Example}

\theoremstyle{remark}
\newtheorem{rmk}[thm]{Remark}

\newcommand{\bbC}{{\mathbb{C}}}

\newcommand{\bbN}{{\mathbb{N}}}

\newcommand{\bbR}{{\mathbb{R}}}

\newcommand{\bbZ}{{\mathbb{Z}}}

\newcommand{\cA}{{\mathcal{A}}}

\newcommand{\cC}{{\mathcal{C}}}

\newcommand{\cE}{{\mathcal{E}}}
\newcommand{\cF}{{\mathcal{F}}}
\newcommand{\cG}{{\mathcal{G}}}

\newcommand{\cI}{{\mathcal{I}}}

\newcommand{\cP}{{\mathcal{P}}}

\newcommand{\del}{{\partial}}

\newcommand{\la}{{\langle}}
\newcommand{\ra}{{\rangle}}

\tikzset{node distance=1.5cm, auto}

\usepackage{babel}

\begin{document}
\maketitle

\begin{abstract}
We construct the augmentation representation. It is a representation of the fundamental group of the link complement associated to an augmentation of the framed cord algebra. This construction connects representations of two link invariants of different types. We also study properties of the augmentation representation.
\end{abstract}

\section{Introduction}
A link is a disjoint union of simple closed curves. A link invariant is an algebraic construction associated to links which is well-defined within each isotopy class. Research on link invariants serves the goal of not only distinguishing links, but also understanding fundamental properties of links and related subjects.

Link invariants appear in various guises. Considering the complement space of a link and taking its fundamental group, we get a group as a link invariant which is known as the link group. In a more complicated form, another algebraic construction we study in this paper is the framed cord algebra, which is a non-commutative algebra generated by paths beginning and ending on the framing longitudes of the link.

They are powerful invariants. In the case of knots, both invariants in their further enhanced forms become complete invariants \cite{W, ENS}, meaning that distinct isotopy classes result in non-isomorphic invariants. However, robust invariants can be impractical to distinguish knots or links, because sometimes it is difficult to show that two groups or two non-commutative algebras are not isomorphic. 

A possible trade-off is to study more computable invariants, such as representations of a group or an algebra. An augmentation is a rank one representation of the framed cord algebra. We want to understand augmentations in terms of representations of the link group, since a group appears simpler than a non-commutative algebra. We approach this goal by constructing the \textit{augmentation representation}.

\begin{thm} Let $L$ be an oriented link with its Seifert framing. Let $\mathrm{Cord}(L)$ be the framed cord algebra and $\pi_L$ be the link group. 

(Theorem-Definition \ref{MainConstruction}) Let $\epsilon: \mathrm{Cord}(L)\rightarrow k$ be an augmentation of the framed cord algebra. By writing $L$ as a braid closure, we construct a representation of the link group $$\rho_\epsilon: \pi_L\rightarrow GL(V_\epsilon).$$

(Theorem \ref{MainThm}) Up to isomorphism, $(\rho_\epsilon,V_\epsilon)$ is well-defined for the augmentation $\epsilon$. In particular, it does not depend on the choice of the braid in the construction.
\end{thm}

The slogan of the construction is ``action by interpolation''. Placing the link as the closure of a braid, we can select a set of standard cords and arrange their augmented values into a square matrix. The column vectors of the matrix span a vector space which is the underlying space of the augmentation representation. A based loop acts on an entry of the matrix by interpolating a standard cord with the based loop, see Figure \ref{Fig:Interpolation}.

\begin{figure}[h]
	\centering
	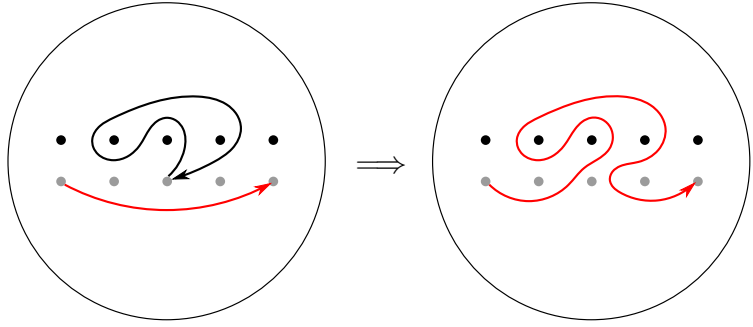
	\caption[]{Action by interpolation. In this example, we start on the left with a based loop (black) in the fundamental group and a standard cord (red). After interpolation we get another framed cord on the right. The action of the based loop takes the augmented value of the standard cord to the augmented value of the interpolated cord.}
	\label{Fig:Interpolation}
\end{figure}

In the next part of the paper, we study some properties of the augmentation representation.

\begin{prop} Let $L$ be an oriented link with its Seifert framing, and $\epsilon: \mathrm{Cord}(L)\rightarrow k$ an augmentation of the framed cord algebra. The augmentation representation $(\rho_\epsilon,V_\epsilon)$ satisfies the following properties.
\begin{itemize}
\item
(Proposition \ref{propSimpleness}) Microlocal simpleness.

For any meridian $m$, there is a $\mathrm{codim}$ $1$ subspace $W\subset V_\epsilon$ such that $\rho_\epsilon(m)|_W = \mathrm{id}_{W}$.

\item
(Proposition \ref{propVanishing}) Vanishing.

Suppose $L_0\subset L$ is a sublink. If $\epsilon(\gamma) = 0$ for either (1) every framed cord $\gamma$ starting from $L_0$, or (2) every framed cord $\gamma$ ending on $L_0$, then $\rho(m_0) = \mathrm{id}_{V_\epsilon}$ for any meridian $m_0$ of $L_0$.
\item 
(Proposition \ref{separability}) Separability.

Suppose $L = L_1\sqcup L_2$ is the union of two sublinks. If $\epsilon(\gamma) = 0$ for all mixed cords between $L_1$ and $L_2$, then $(\rho_\epsilon, V_\epsilon)$ is a direct sum of two representations $(\rho_1, V_1)$ and $(\rho_2, V_2)$, where each $(\rho_i,V_i)$ is a representation of $\pi_{L_i}$.
\end{itemize}
\end{prop}

These properties can be reinterpreted in terms of microlocal sheaf theory. In microlocal theory, one considers the ``micro-support'' in the cotangent bundle, generalizing the usual notion of the support of a sheaf. The augmentation representation is defined over the link group, which is equivalent to a locally constant sheaf on the link complement. Taking the underived push-forward, we obtain a sheaf $\cF$ microsupported within the conormal bundle of the link.

The three properties of $(\rho_\epsilon,V_\epsilon)$ can be rephrased in terms of $\cF$. The first property, $(\rho_\epsilon,V_\epsilon)$ being microlocally simple, is equivalent to say that $\cF$ is a simple sheaf along its micro-support (in the convention of \cite{KS}), or microlocal rank $1$ along its micro-support (in the convention of \cite{STZ}). The vanishing property states a sufficient condition when $\cF$ is microsupported outside the conormal of the sublink $L_0$. The separability states a sufficient condition when $\cF$ splits into two subsheaves, where each subsheaf is microsupported along a sublink.

\smallskip
The definition and properties of the augmentation representation outreach to several directions. First, the framed cord algebra can be thought of as being generated by a subcategory of the fundamental groupoid of the link complement, and one expects from functoriality that representations should be pulled back in some sense. The first such construction appeared in \cite{Co}, where $L=K$ is a knot and $\epsilon$ admits a mild restriction. Cornwell showed that these augmentations are bijective to irreducible ``KCH'' representations. When $L$ is a link, the framed cord algebra is no longer generated by $\pi_L$, and must take a more general form. The current construction works more generally, including mixed cords arising from links.

A character variety is a moduli space of representations of a finitely generated group. The character variety of the link group is a space of flat bundles on the link complement, and it plays an important role in knot theory and three manifolds. The $SL(2)$-character variety has been intensively studied. We mention \cite{CS,KM} and the $A$-polynomial discussed in the next paragraph. The $SL(n)$-character variety in higher ranks becomes increasingly complicated \cite{AH, GTZ, GW, HMP, MP}.  Our results show that the ``full augmentation variety'', which is the moduli space of augmentations, cuts off a closed subvariety in the link group character variety. It suggests a direction of future research to understand the subvariety characterized by microlocally simple representations.

Introduced by Cooper et al. \cite{CCGLS}, the $A$-polynomial of a knot defines a complex plane curves (which we call the ``$A$-variety'' for now) which parametrizes the $SL(2)$-character variety projected to a torus determined by the peripheral subgroup. It can be used to detect knots \cite{DG, BZ, NiZh}. The $A$-polynomial is closely related to the augmentation polynomial, whose vanishing locus is the moduli space of augmentations projected to the same torus, or the ``augmentation variety''. When $L=K$ is a knot, both polynomials have two variables, denoted as $A_{K}(\lambda,\mu)$ and $Aug_K(\lambda,\mu)$. Ng proved that $A_{K}(\lambda,\mu)$ divides $Aug_K(\lambda,\mu^2)$ \cite{Ng3}. It is conjectured in \cite{AENV} that the augmentation polynomial with a larger coefficient ring is a generalization of the $A$-polynomial, and it produces a new notion of mirror symmetry. To see the relation between two polynomials using our result, augmentation representations are microlocally simple, and a generic rank $2$ microlocally simple representation is in one-to-one correspondence with $SL(2)$-representations.  When $L$ is a link, we no long have polynomials since the ideals of vanishing function are not principal, but a similar result holds for the same reason --- the $A$-variety is a closed subvariety of the augmentation variety.

Finally the augmentation representation builds up to the correspondence between augmentations and sheaves for links. The correspondence was proven for Legendrian links \cite{NRSSZ} and connected Legendrian surfaces defined from knot conormals \cite{Ga2} or cubic graphs (by Sackel in appendix of \cite{CM}). For a brief motivation, augmentations of a framed cord algebra correspond to augmentations of the dga for the conormal tori \cite{Ng3}, whose geometric counterparts are Lagrangian fillings in the sense of SFT \cite{El, EGH}. Fillings are sheaves through microlocalization \cite{Na, NZ}, and further determine link group representations via the Radon transform \cite{Ga1}. For knots, the correspondence in both ways are constructed. For links, the construction is more involved because of mixed cords. In this paper, we focus on the geometric origin of the theory and construct the augmentation representation. As explained in Figure \ref{Fig:Interpolation}, the action of the link group has a geometric meaning, namely based loops act by their interpolations in framed cords. This construction makes a direct connection from augmentations of a framed cord algebra to simple sheaves. 

This paper is organized in a simple fashion. We construct the augmentation representation is Section \ref{Sec:AugRep} and study its properties in Section \ref{Sec:Properties}. Section \ref{Sec:micro} is a short review of the microlocal sheaf theory.

\noindent\textbf{Acknowledgements.} We thank Lenhard Ng, Stéphane Guillermou, Eric Zaslow for helpful discussions and valuable comments. We thank the referee for important suggestions. This work is supported by ANR-15-CE40-0007 ``MICROLOCAL''.

\section{Augmentation representation}\label{Sec:AugRep}
The goal of this section is to construct the augmentation representation. It is a link group representation associated to an augmentation. We first introduce some preliminary concepts, including the framed cord algebra and its augmentations, and then construct the augmentation representation.

Let $k$ be a commutative field. It is the ground field where we will define augmentations and group representations. Throughout the paper, we set $X= \bbR^3$ or $S^3$. Let $L$ be an oriented $r$-component link in $X$. We label its components as 
$$L = K_1\sqcup K_2\sqcup  \dotsb \sqcup K_r.$$

Suppose $p:[0,1]\rightarrow X$ is a path or a loop, we write $p^{-1}$ for the reversed path: 
$$p^{-1}(t) = p(1-t).$$
We denote by $p_1\cdot p_2$ the concatenation of two composable paths. Namely if $p_1(1) = p_2(0)$, we define
$$p_1\cdot p_2(t)  = 
\begin{cases}
p_1(2t), & 0\leq t\leq 1/2, \\
p_2(2t-1), & 1/2\leq t\leq 1.
\end{cases}
$$
We work with paths up to homotopy, hence the concatenation induces an associative product.

\subsection{Braids}\label{Sec:Braid}
To construct an augmentation representation, we need to represent the link as the closure of a braid. A braid can be expressed as an element in Artin's braid group of $n$ strands:
$$Br_n  = \la\, \sigma_1^{\pm 1},\dotsb, \sigma_{n-1}^{\pm 1} \,|\, \sigma_i\sigma_{i+1}\sigma_i = \sigma_{i+1}\sigma_i\sigma_{i+1}, \textrm{ and } \sigma_j\sigma_k = \sigma_k\sigma_j \textrm{ for } |j-k|\geq 2\,\ra.$$

Geometrically, a braid is a collection of strands in a solid cylinder, with endpoints fixed on boundary disks. Braids can be realized as a mapping class group. Let $D$ be an oriented disk with sufficiently large radius, equipped with two-tuple coordinates $(-,-)$. Let $y_1 = (1,0),\dotsb, y_n = (n,0)$ be $n$ marked points, and $D^\circ \subset D$ be the $n$-punctured disk with marked points removed. The braid group $Br_n$ is isomorphic to the mapping class group $\textrm{MCG}(D^\circ) := \pi_0(\textrm{Diff}\,^+(D^\circ))$, where $\textrm{Diff}\,^+(D^\circ)$ is the topological group of orientation preserving diffeomorphisms of $D^\circ$. For each $[h]\in \textrm{MCG}(D^\circ)$, we can extend $h$ to a homeomorphism $\tilde h: D\rightarrow D$. Let $H: D\times [0,1]\rightarrow D$ be a $C^0$ isotopy between $\textrm{id}_D$ and $h$, namely $H(-,0) = \textrm{id}_D$, $H(-,1) = h$, and $H(-,t): D\rightarrow D$ is a homeomorphism for each $t\in [0,1]$. Then $H^{-1}(\{y_1,\dotsb,y_n\})\subset D\times [0,1]$ is the braid associated to $[h]\in \textrm{MCG}(D^\circ) \cong Br_n$.

\begin{eg}\label{Whiteheadbraid} In Figure \ref{Fig:Whiteheadbraid}, we plot a $3$-strand braid whose braid word is
$$\sigma_1^2\sigma_2^2\sigma_1^{-1}\sigma_2^{-2}.$$
The closure of the braid is the two-component Whitehead link.
\begin{figure}[h]
	\centering
	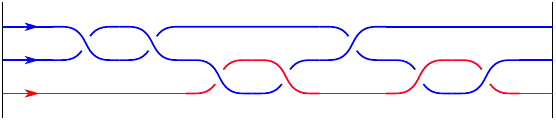
	\caption[]{A braid whose closure is the Whitehead link.}
	\label{Fig:Whiteheadbraid}
\end{figure}
\end{eg}
The \textit{configuration disk} $D$ with $n$ marked points will be extensively used in the upcoming constructions. In practice, one can plot the figure in a more symmetric fashion. We keep the coordinates for an easier description.

Given a braid $B \in D\times [0,1]$, one can close the solid torus by identifying $D\times \{0\}$ and $D\times \{1\}$. Since the marked points are fixed on the boundaries, the braid can be closed to a one dimensional compact submanifold $\la B \ra$ in the solid torus. If we take a sectional disk, in particular at the place where we glue, we obtain the configuration disk where the marked points equal to $D\cap \la B \ra$. If the solid torus is further embedded in $X = \bbR^3$ or $S^3$ as the tubular neighborhood of an unknot with $0$ framing, we obtain a link $L  = \la B \ra\subset X$. A braid carries a natural orientation induced from the orientation of $[0,1]$. Hence the braid closure is naturally an orientated link. Alexander's theorem asserts that every oriented link can be expressed (not in a unique way) as a braid closure. See \cite{Al} for the original construction, or \cite{Ya,Vo} for the improved Yamada-Vogel algorithm.

\begin{figure}[h]
	\centering
\begingroup%
  \makeatletter%
  \providecommand\color[2][]{%
    \errmessage{(Inkscape) Color is used for the text in Inkscape, but the package 'color.sty' is not loaded}%
    \renewcommand\color[2][]{}%
  }%
  \providecommand\transparent[1]{%
    \errmessage{(Inkscape) Transparency is used (non-zero) for the text in Inkscape, but the package 'transparent.sty' is not loaded}%
    \renewcommand\transparent[1]{}%
  }%
  \providecommand\rotatebox[2]{#2}%
  \newcommand*\fsize{\dimexpr\f@size pt\relax}%
  \newcommand*\lineheight[1]{\fontsize{\fsize}{#1\fsize}\selectfont}%
  \ifx\svgwidth\undefined%
    \setlength{\unitlength}{277.6bp}%
    \ifx\svgscale\undefined%
      \relax%
    \else%
      \setlength{\unitlength}{\unitlength * \real{\svgscale}}%
    \fi%
  \else%
    \setlength{\unitlength}{\svgwidth}%
  \fi%
  \global\let\svgwidth\undefined%
  \global\let\svgscale\undefined%
  \makeatother%
  \begin{picture}(1,0.33429395)%
    \lineheight{1}%
    \setlength\tabcolsep{0pt}%
    \put(0,0){\includegraphics[width=\unitlength,page=1]{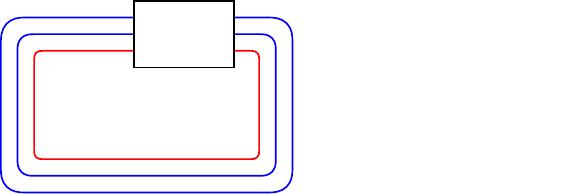}}%
    \put(0.30072804,0.2634702){\color[rgb]{0,0,0}\makebox(0,0)[lt]{\lineheight{1.25}\smash{\begin{tabular}[t]{l}$B$\end{tabular}}}}%
    \put(0,0){\includegraphics[width=\unitlength,page=2]{ConfigDisk.pdf}}%
    \put(0.11681782,0.17537031){\color[rgb]{0,0,0}\makebox(0,0)[lt]{\lineheight{1.25}\smash{\begin{tabular}[t]{l}$D$\end{tabular}}}}%
    \put(0,0){\includegraphics[width=\unitlength,page=3]{ConfigDisk.pdf}}%
    \put(0.67013189,0.04856916){\color[rgb]{0,0,0}\makebox(0,0)[lt]{\lineheight{1.25}\smash{\begin{tabular}[t]{l}$D$\end{tabular}}}}%
    \put(0,0){\includegraphics[width=\unitlength,page=4]{ConfigDisk.pdf}}%
    \put(0.7795389,0.39729744){\color[rgb]{0,0,0}\makebox(0,0)[lt]{\begin{minipage}{0.20172911\unitlength}\raggedright \end{minipage}}}%
  \end{picture}%
\endgroup%

	\caption[]{The configuration disk of a braid closure. The dashed line is where we take the disk.}
	\label{Fig:ConfigDisk}
\end{figure}

Suppose an $r$-component link $L = K_1\sqcup \dotsb \sqcup K_r$ is the closure of an $n$-strand braid $B$. The strands in $B$ can be indexed by the first coordinate in the disk $D$. We define the \textit{component function} to be a map
\begin{equation}\label{IndexFun}
\{ - \}: \{1,\dotsb, n\}\rightarrow \{1,\dotsb, r\}, 
\end{equation}
such that the strand $t$ of the braid belongs to the component $K_{\{t\}}$.

The strand index admits a natural linear ordering by the $x$-coordinate of the marked points in the configuration disk. A partition of a linearly ordered set is \textit{ordered}, if for any two parts $P_1,P_2$ of the partition, one has either
$$a< b \textrm{  for all } a\in P_1, b\in P_2, \quad \textrm{or} \quad a> b \textrm{  for all } a\in P_1, b\in P_2.$$

\begin{lem}\label{sortingLem}
If $L$ is the closure of an $n$-strand braid, then it can be represented by  an $n$-strand braid which admits an ordered partition where strands lie in the same part if and only if they belong to the same component of $L$.
\end{lem}
\begin{proof}
Suppose $L  = \la B\ra$ for some $n$-strand braid $B$. For any half twist $\sigma_i$, there is $\la B \ra = \la \sigma_i \sigma_i^{-1} B \ra = \la  \sigma_i^{-1} B \sigma_i\ra$. The conjugated braid $\sigma_i^{-1} B \sigma_i$ is also an $n$-strand braid, with $i$-th and $(i+1)$-th strands switched. We can conjugate the braid finitely many times until we obtain the desired braid.
\end{proof}

\begin{rmk}
By Lemma \ref{sortingLem}, we can assume there exist integers
$$0 = n_0 < n_1 < \dots < n_{r-1}< n_r = n,$$
such that the closure of the strands $\{n_{i-1}+1, n_{i-1}+2, \dotsb, n_i \}$ is the component $K_i$ of $L$. In other words, we can assume the component function is non-decreasing.
\end{rmk}

\subsection{Link group}
Let $X = \bbR^3$ or $S^3$. Let $L \subset X$ be an oriented $r$-component link, $L = K_1\sqcup \dotsb \sqcup K_r$. The link group $\pi_L$ is the fundamental group of the link complement, i.e. $\pi_L = \pi_1(X\setminus L)$. The link group has the following properties:

\begin{itemize}
\item
It is finitely generated by meridians, and finitely presented.
\item
There are $r$ conjugacy classes, labelled by components of the link.
\end{itemize}

These properties follow easily from the Wirtinger presentation of the link group, (for example, see \cite{Ro}). We recall the construction. Thinking of an oriented link by it two-dimensional diagram with under-crossings, in a generic position, the diagram has finitely many arcs and under-crossings. We take the base point of the fundamental group far away from the plane. Each arc corresponding to a loop in the fundamental group, which travels from the base point to the plane, wraps around the arc and then travels back. The orientation of the loop is determined by the orientation of the link. Note this loop is a meridian, which is by definition the boundary of a disk intersecting the link transversely at a point. Finally, each under-crossing imposes a conjugation relation among meridian generators, giving the Wirtinger presentation.

When the link $L$ is the closure of an $n$-strand braid, the link group $\pi_L$ can be generated by $n$ meridians, (though $n$ may not be the minimum number of meridian generators). To see this, we scan the braid diagram from left to right. At the beginning, each strand determines a meridian, and we denote it by $m_t$ for $1\leq t\leq n$. The braid is given by a word of half twists. Each half twist introduces an under-crossing, and therefore a new meridian in the Wirtinger presentation. The new meridian can be expressed as a word of previous meridians. Iterating the procedure, we see that the set of meridians $\{m_t\}_{1\leq t\leq n}$ generates the whole link group. In the rest of the paper, we will call 
$$\{m_t\}_{1\leq t\leq n}$$
the \textit{generating set of meridians} of an $n$-strand braid closure.

If we set the base point of the link group to be $(0, -\delta)\in D\subset X$, where $\delta$ is a small positive real number, then the meridian generators $\{m_t\}_{1\leq t\leq n}$ can be plotted on the configuration disk. Namely, $m_t$ is the loop in $D$ wrapping around the marked point $y_t = (t,0)$.

To be compatible with the convention in \cite{CELN}, we fix the orientation of the knot in the configuration disk pointing inward to the paper, and the meridian generators wrap clockwisely around marked points. See Figure \ref{Fig:ConfigurationDisk} for an example.

\begin{eg}
The planar diagram of the Whitehead link is plotted in Figure \ref{Fig:Whiteheadlink}.

\begin{figure}[h]
	\centering
	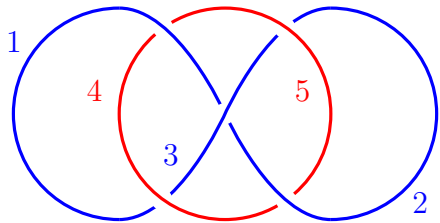
	\caption[]{The Whitehead link.}
	\label{Fig:Whiteheadlink}
\end{figure}
There are five strands, giving five meridian generators in the Wirtinger presentation of the link group. Each crossing in the diagram gives a conjugation relation. Therefore the link group is isomorphic to
$$\pi_L = \la m_1, m_2,m_3, m_4,m _5 \ra / \sim, $$
modulo relations
\begin{align*}
m_1\cdot m_5 &= m_4\cdot m_1,\quad m_3\cdot m_5 = m_5\cdot m_2,\quad m_1\cdot m_3 = m_3\cdot m_2, \\
&\quad m_3\cdot m_4 = m_4\cdot m_1,\quad m_4\cdot m_2 = m_2\cdot m_5.
\end{align*}

We have seen that the Whitehead link can be expressed as the closure of a $3$-strand braid, as in Example \ref{Whiteheadbraid}. Therefore the link group can be generated by three meridians instead of five. As for the Wirtinger presentation we just computed, we can write $m_3,m_5$ in terms of $m_1,m_2, m_4$, reducing the number of meridian generators to three.
\end{eg}

\subsection{Framed cord algebra}

The cord algebra first appeared in \cite{Ng1, Ng2}. The framed version was introduced in \cite{Ng3}, which models the degree zero knot contact homology \cite{EENS}.

Suppose $K\subset X$ is an oriented knot, and $n(K)$ is a small tubular neighborhood of $K$. A \textit{framing} of $K$ is a push off the knot to the boundary of the tubular neighborhood. In other words, a curve $\ell\subset \del(n(K))$ whose homology class in $H_1(n(K))$ agrees with $[K]\in H_1(n(K))$. A Seifert surface is an oriented surface $S$ with $\del S = K$. The \textit{Seifert framing} is $\ell = S\cap n(K)$, which is up to homotopy independent from the choice of the Seifert surface. The Seifert framing has zero linking number with $K$.

A framing of a link $L = K_1\sqcup \dotsb \sqcup K_r$ is the choice of a framing $\ell_i$ for each component $K_i$. We decorate each $\ell_i$ with a marked point $\ast_i\in \ell_i$. We write $\ast := \{\ast_1,\dotsb, \ast_r\}$, and 
$$L' = \ell_1\sqcup \dotsb \sqcup \ell_r.$$
We write $(L,L')$ for a framed link, namely an oriented link with a choice of a framing.

Upcoming, we define the framed cord algebra of a framed link, which is a mild generalization of the definition for knots as in {\cite[Definition 2.5]{CELN}}.

\begin{defn}
Suppose $(L, L')\subset X$ is a framed link. 

A framed cord of $L$ is a continuous map $c: [0,1]\rightarrow X\setminus L$ such that $c(0), c(1) \in L \setminus \ast$. Two framed cords are homotopic if they are homotopic through framed cords. We write $[c]$ for the homotopy class of the cord $c$.

We now construct a non-commutative unital ring $\cA$ as follows: as a ring, $\cA$ is freely generated by homotopy classes of framed cords and extra generators $\lambda_i^{\pm 1}, \mu_i^{\pm 1}$, $1\leq i\leq r$, modulo the following relations.

For any $1\leq i \leq r$,
$$
\lambda_i \cdot \lambda_i^{-1} = \lambda_i^{-1}\cdot \lambda_i = \mu_i\cdot \mu_i^{-1} = \mu_i^{-1}\cdot \mu_i = 1,
$$
and
$$
\lambda_i \cdot \mu_i = \mu_i \cdot \lambda_i. 
$$
Thus $\cA$ is generated as a $\bbZ$-module by non-commutative words in homotopy classes of cords and powers of $\lambda_i$ and $\mu_j$. The powers of the $\lambda_i$ and $\mu_i$ commute with each other, but do not with any cords.

The framed cord algebra is the quotient ring 
$$\textrm{Cord}(L, L') = \cA/\cI,$$
where $\cI$ is the two-sided ideal of $\cA$ generated by the following relations:

\pagebreak

\begin{itemize}
\item
(Normalization)
\begin{figure}[h]
	\centering
	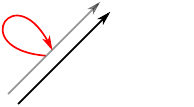
	\label{Fig:Normalization}
\end{figure}
\item
(Meridian)
\begin{figure}[h]
	\centering
	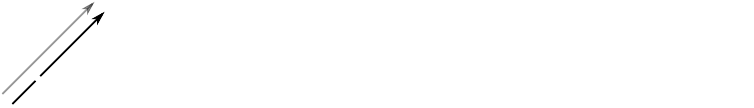
	\label{Fig:Meridian}
\end{figure}
\item
(Longitude)
\begin{figure}[h]
	\centering
	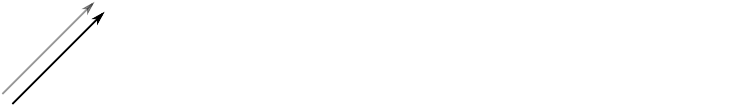
	\label{Fig:Longitude}
\end{figure}

\item (Skein relations)
\begin{figure}[h]
	\centering
	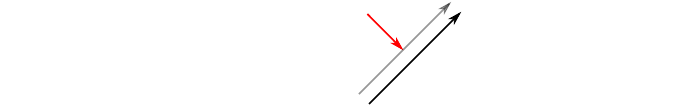
	\label{Fig:Skeinrelation}
\end{figure}

\end{itemize}

\begin{rmk}
The definition of a framed cord algebra, at first glance, depends on the choice of decorations $\ast_i$. Different choices of decorations give isomorphic algebras. Consider the following argument. Since generators are framed cords up to cord homotopy, we can choose a base point on each framing and homotope the endpoints of each cord to the base point. For example, we can choose base points to be a positive push off of marked points $\ast_i$ along each component of the framing. For two different sets of marked points, we can homotope one set of marked points to the other set along the framing. This homotopy induces an isomorphism between the framed cord algebras.
\end{rmk}

\begin{rmk}\label{rmkframing}
Given two choices of framings $L',L''$, their framed cord algebras are isomorphic as $\bbZ$-algebras. We can construct
$$\textrm{Cord}(L,L') \cong \textrm{Cord}(L, L''),$$
by keeping $\mu_i$ and framed cords, and sending $\lambda_i$ to $\lambda_i\cdot \mu_i^{\textrm{lk}_i(L,L'')-\textrm{lk}_i(L,L')}$. Here 
$\textrm{lk}_i(L,-)$ is the linking number between $K_i$ and its framing.

If $L'$ is the Seifert framing, we simply write the framed cord algebra as $\mathrm{Cord}(L)$.
\end{rmk}

\begin{rmk}
Let $\textrm{Cord}^c(L,L')$ be a quotient of $\textrm{Cord}(L,L')$, where $\lambda_i,\mu_i$ commute with everything. It is a $\bbZ[\lambda^{\pm1}_1,\mu^{\pm1}_1,\dotsb, \lambda^{\pm1}_r, \mu^{\pm1}_r]$-algebra. In particular we have $\lambda_i^{\pm1}, \mu_i^{\pm1}$ commute with all framed cords, and moreover for $1\leq i,j\leq r$,
$$\lambda_i \cdot \lambda_j = \lambda_j\cdot \lambda_i, \quad
\mu_i \cdot \mu_j = \mu_j\cdot \mu_i,\quad \lambda_i \cdot \mu_j = \mu_j\cdot \lambda_i.$$

Following \cite{Ng3}, $\textrm{Cord}^c(L,L')$ is isomorphic to the degree $0$ homology of a differential $\bbZ[\lambda^{\pm1}_1,\mu^{\pm1}_1,\dotsb, \lambda^{\pm1}_r, \mu^{\pm1}_r]$-algebra of the conormal tori of the link.

For different framings $L',L''$, there is an isomorphism $\textrm{Cord}^c(L,L') \cong \textrm{Cord}^c(L,L'')$. Note it is an isomorphism of $\bbZ$-algebras, but not of $\bbZ[\lambda^{\pm1}_1,\mu^{\pm1}_1,\dotsb, \lambda^{\pm1}_r, \mu^{\pm1}_r]$-algebras.
\end{rmk}

\end{defn}

\begin{rmk}\label{3fcas}
There are various versions of the (framed) cord algebra in literature, $\cC_L$ \cite{Ng3}, $\textrm{Cord}^c(L)$ \cite{CELN}, and $\cP_K$ \cite{Ng4}. Each cord algebra is a quotient ring of the free non-commutative algebra generated by some version of cords over the same coefficient ring $\bbZ[\lambda^{\pm1}_1,\mu^{\pm1}_1,\dotsb, \lambda^{\pm1}_r, \mu^{\pm1}_r]$. In each version, the free algebra quotients out four relations:  normalizations, meridian relations, longitude relations, and skein relations. We compare their definitions with a focus on skein relations.
\begin{enumerate}
\item
\cite{Ng3} $\cC_L$ is defined for a link $L$. Choose a marked point $\ast_i\in K_i$ for each $1\leq i\leq r$. A cord is a continue map $c: [0,1]\rightarrow X\setminus L$, with $c(0), c(1)\in L\setminus \{\ast_1\dotsb, \ast_r\}$. The skein relations are

\begin{figure}[h]
	\centering
	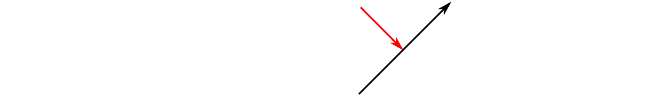
	\label{Fig:SkeinrelationCKrmk}
\end{figure}

\item
\cite{CELN} $\textrm{Cord}^c(L,L')$ is defined for a framed link $(L,L')$.  Choose a maked point $\ast_i\in \ell_i$ for each $1\leq i\leq r$. A cord is a continue map $c: [0,1]\rightarrow X\setminus L$, with $c(0), c(1)\in L'\setminus \{\ast_1\dotsb, \ast_r\}$. The skein relations are

\begin{figure}[h]
	\centering
	\input{Skeinrelation.pdf_tex}
	\label{Fig:Skeinrelationrmk}
\end{figure}

\item
\cite{Ng4} $\cP_K$ is defined for a knot $K$, namely $r =1$. A cord $[h]\in \cP_K$ is represented by a based loop $h\in \pi_K$. In particular, we forget the group structure in $\pi_K$ and view it as a set, and then take the set elements as free generators. Let $m$ be the meridian in the peripheral subgroup of $\pi_K$. The skein relations are
$$[h_1\cdot h_2] = [h_1\cdot m \cdot h_2] + [h_1][h_2].$$
Note that the $\cP_K$ version is only defined for knots, but not for links.

\end{enumerate}
If $L'$ is the Seifert framing, then $(1) = (2)$. If in addition that $L=K$ is a single component knot, then $(1) = (2) = (3)$. Namely, there are $\bbZ[\lambda^{\pm1}_1,\mu^{\pm1}_1,\dotsb, \lambda^{\pm1}_r, \mu^{\pm1}_r]$-algebra isomorphisms
$$\cC_K \cong \textrm{Cord}^c(K) \cong \cP_K.$$
From $\cC_K$ to $\textrm{Cord}^c(K)$, one can push the endpoints of a cord off the link to the framing. There are different choices of the push offs, and they are related by multiplying copies of $\mu_{i}$ depending on the linking number. From $\textrm{Cord}^c(K)$ to $\cP_K$, one can choose a base point on each framing, then move the endpoints of each framed cord to the base point along the orientation of the framing.
\end{rmk}

Cords in a framed cord algebra can be classified into pure and mixed cords depending on the points where the cord starts and ends.

\begin{defn}
Let $L$ be a link.

Suppose $L_0 \subset L$ is a sublink. A cord is a \textit{pure cord} of $L_0$ if it starts and ends on the framing of $L_0$.

Suppose $L_1, L_2 \subset L$ are two disjoint sublinks. A cord is a \textit{mixed cord} between $L_1$ and $L_2$ if either (1) it starts on the framing of $L_1$ and ends on the framing of $L_2$, or (2) it starts on the framing of $L_2$ and ends on the framing of $L_1$.

If we simply say a pure cord or a mixed cord without decorations, it means the underlying sublink is a knot.\end{defn}

Recall the configuration disk $D$ with $n$-marked points $\{y_1,\dotsb,y_n\}$ can be regarded as a locally closed submanifold in $X$ with $\{y_1,\dotsb,y_n\} = D\cap L$. We can perturb the framing such that $D\cap L' = \{x_1, x_2,\dotsb, x_n \} = \{(1, -\delta), (2, -\delta),\dotsb, (n, -\delta)\}$, where $\delta$ is a small positive number. Recall we had a marked point $\ast_i$ on each $\ell_i$, where $1\leq i\leq r$. We assume that $\{\ast_i\}_{1\leq i\leq r}$ and $\{x_j\}_{1\leq j\leq n}$ are all distinct.

The orientation of the link $L$ induces an orientation of the framing $L'$. After a suitable perturbation, the framing $\ell_{\{i\}}$ can be viewed as a (framed) cord, starting and ending at $x_i$. We call it a \textit{framing cord}, and denote it by $[\ell_{\{i\}}^{(i)}]$.

Recall that we set the based point of $\pi_L$ at $x_0 = (0, -\delta)$. Define the \textit{capping path} $p_i$ to be the linear path from $x_0$ to $x_i$, for $i=1,\dotsb, n$.

If we move the based point of $\pi_L$ to one of $x_i$, a meridian generator $m_t$ can be viewed as a framed cord, which we term as a \textit{meridian cord}. A meridian based at different longitudes are different framed cords in $\mathrm{Cord}(L)$. Hence it is necessary to remember the base point. Let $m_t^{(i)}$ be the meridian cord based at $x_i$ wrapping around $y_t$. Then,
$$m_t = p_i \cdot m_t^{(i)}\cdot p_i^{-1}.$$

\begin{defn}
Define the \textit{standard cord} to be $\gamma_{ij} := p_i^{-1} \cdot p_j$, for $i,j \in\{1, \dotsb, n\}$.
\end{defn}
Standard cords satisfy the following relations:
\begin{itemize}
\item
$\gamma_{ij}\cdot \gamma_{ji} = \gamma_{ii} = e_{\{i\}}$, where $e_{\{i\}}$ is the trivial cord on $\ell_{\{i\}}$,
\item
$\gamma_{ij}\cdot \gamma_{jk} = \gamma_{ik}$.
\end{itemize}

Meridian cords based at different points are related by
$$m_t^{(i)} = \gamma_{ij}\cdot m_t^{(j)} \cdot \gamma_{ji}.$$

\begin{figure}[h]
	\centering
	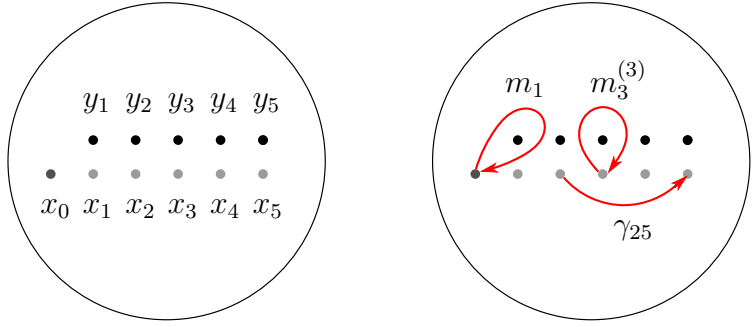
	\caption[]{A configuration disk (left), and examples of meridian cords and a standard cord (right).}
	\label{Fig:ConfigurationDisk}
\end{figure}

The relations among framed cords on the configuration disk can be written in the following way. Let $c_{ij}$ denote a framed cord from $x_i$ to $x_j$. 
\begin{itemize}
\item
Normalization
$$[\gamma_{ii}] = 1 -\mu_{\{i\}},$$
\item
Meridian 
$$[m_i^{(i)}\cdot c_{ij}] = \mu_{\{i\}}[c_{ij}],\quad  [c_{ij} \cdot m_j^{(j)}] = [c_{ij}]\mu_{\{j\}},$$
\item
Longitude $$[\ell^{(i)}_{\{i\}}\cdot c_{ij}] = \lambda_{\{i\}} [c_{ij}], \quad [c_{ij}\cdot \ell^{(j)}_{\{j\}}] = [c_{ij}]\lambda_{\{j\}},$$
\item
Skein relations
\begin{equation}\label{ConfDsr}
[c_{it} \cdot c_{tj}] = [c_{it} \cdot m_t^{(t)}\cdot c_{tj}] + [c_{it}][c_{tj}].
\end{equation}
\end{itemize}
The longitude relation is not visible in the configuration disk, but we write it down anyways. Note a longitude cord can be written as a word of meridian generators in the configuration disk. For most of calculations over the disk, it suffices to use the other three relations.

\begin{lem}\label{StdCordsGen}
Standard cords generate the framed cord algebra.
\end{lem}
\begin{proof}
By generation, we mean that every element in $\mathrm{Cord}(L)$ can be written as a sum of words consisting of standard cords and $\lambda_i^{\pm1}, \mu_i^{\pm1}$. We start by proving two claims.

\smallskip
(1) Every framed cord is cord homotopic to one from $x_i$ to $x_j$ for some $i,j \in\{1,\dotsb, n\}$.

Suppose a cord $[c]$ starts at $x\in\ell_s\setminus \{\ast_s\}$. There exists $x_i$ such that $\{i\}= s$. Take a path $p\subset \ell_s$ connecting $x_i$ and $x$, without passing through point $\ast_s$. The concatenation $p\cdot c$ is homotopic to $c$, and the staring point of $p\cdot c$ is $x_i$. Similarly one can homotope the end point to some $x_j$.

\smallskip
(2) Every framed cord from $x_i$ to $x_j$ is cord homotopic to the concatenation of a loop $h^{(i)} \in \pi_1(X\setminus K, x_i)$ and the standard cord $\gamma_{ij}$.

Suppose $[c_{ij}]$ is a framed cord from $x_i$ to $x_j$. We set $h^{(i)} = c_{ij}\cdot \gamma_{ji}$, then $[c_{ij}] = [c_{ij}\cdot \gamma_{ji} \cdot \gamma_{ij}] = [h^{(i)} \cdot \gamma_{ij}]$, proving the claim.

\smallskip
Now we prove the lemma. By statements (1) and (2), it suffices to prove the assertion for a framed cord in the form of $[c] = [h^{(i)} \cdot \gamma_{ij}]$. Since $h^{(i)} \in \pi_1(X\setminus L, x_i)$, it is generated by meridians. We prove by induction on the word length of $h^{(i)}$. The initial step is trivial, since $[c] = [\gamma_{ij}]$ is already a standard cord.

Suppose the induction hypothesis that $[h^{(i)}\cdot \gamma_{ij}]$ is generated by standard cords holds for any loop $h^{(i)}\in \pi_L$ with any base point $x_i$ for $1\leq i\leq n$ and with word length less or equal to $s \in \bbN$, we need to show that the induction hypothesis also holds for word length $s+1$. Namely, $[(m_t^{(i)})^{\pm 1}\cdot h^{(i)}\cdot \gamma_{ij}]$ can be generated by standard cords. There is
\begin{align*}
[m_t^{(i)}\cdot h^{(i)} \cdot \gamma_{ij}] 
	&= [\gamma_{it} \cdot m_t^{(t)} \cdot \gamma_{ti} \cdot h^{(i)} \cdot \gamma_{ij}]  \\
	&= [\gamma_{it} \cdot \gamma_{ti} \cdot h^{(i)} \cdot \gamma_{ij}] - [\gamma_{it}] [ \gamma_{ti} \cdot h^{(i)} \cdot \gamma_{ij}]  \\
	&= [h^{(i)} \cdot \gamma_{ij}] - [\gamma_{it}] [ \gamma_{ti} \cdot h^{(i)} \cdot \gamma_{it}\cdot \gamma_{ti}\cdot \gamma_{ij}] \\
	&= [h^{(i)} \cdot \gamma_{ij}] - [\gamma_{it}] [ h^{(t)} \cdot \gamma_{tj}].
\end{align*}
The second equality uses the skein relation, and the other equalities are cord identities. Similarly we have
\begin{align*}
[(m_t^{(i)})^{-1}\cdot h^{(i)} \cdot \gamma_{ij}] 
	&= [\gamma_{it} \cdot (m_t^{(t)})^{-1} \cdot \gamma_{ti} \cdot h^{(i)} \cdot \gamma_{ij}]  \\
	&= [h^{(i)} \cdot \gamma_{ij}] + [\gamma_{it}] [(m_t^{(t)})^{-1} \cdot \gamma_{ti} \cdot h^{(i)} \cdot \gamma_{ij}]  \\
	&= [h^{(i)} \cdot \gamma_{ij}] + [\gamma_{it}] \cdot \mu_{\{t\}}^{-1} \cdot [ h^{(t)} \cdot \gamma_{tj}].
\end{align*}
 The last equation uses the meridian relation. We complete the induction, as well as the proof of the lemma.
\end{proof}

\begin{rmk}
It follows the lemma that meridian cords are generated by standard cords.
Suppose $[m_t^{(b)}]$ is a meridian cord wrapping around strand $t$, and based at $x_b$. That is to say, the underlining loop is $m_t \in \pi_1(X\setminus L, x_b)$. Then
\begin{align*}
[m_t^{(b)}] = [\gamma_{bb}]  -  [\gamma_{bt}][\gamma_{tb}], \qquad 
[(m_t^{(b)})^{-1}] = [\gamma_{bb}]  +  [\gamma_{bt}] \cdot \mu_{\{t\}}^{-1} \cdot [\gamma_{tb}]. 
\end{align*}
Here $\{t\}$ is the component function defined in (\ref{IndexFun}). Note that a meridian cord depends on the base point as well as the homotopy class of the underlying loop.
 
\end{rmk}

\subsection{Augmentations}
Let $L$ be an oriented link and let $\mathrm{Cord}(L)$ be its framed cord algebra.  An augmentation $\epsilon$ of $\mathrm{Cord}(L)$ is a unit preserving algebra morphism
$$\epsilon: \mathrm{Cord}(L)\rightarrow k,$$
where $k$ is any commutative field.

Recall we have a variant framed cord algebra $\textrm{Cord}^c(L)$ where $\lambda_i^{\pm1},\mu_i^{\pm1}$ commute with cords. We can similarly define its augmentation as a unit preserving algebra morphism
$$\epsilon: \mathrm{Cord}^c(L)\rightarrow k.$$
Since $k$ is commutative, the image of any non-commutative generators in either $\mathrm{Cord}(L)$ or $\mathrm{Cord}^c(L)$ becomes commutative as elements in $k$. Therefore, the set of augmentations for $\mathrm{Cord}(L)$ and $\mathrm{Cord}^c(L)$ are canonically isomorphic.

\begin{rmk}
In Remark \ref{rmkframing} we explained that different framings give isomorphic $\bbZ$-algebras but not $\bbZ[\lambda^{\pm1}_1,\mu^{\pm1}_1,\dotsb, \lambda^{\pm1}_r, \mu^{\pm1}_r]$-algebras. Hence there is a bijection between augmentations with respect to two different framings. However, if we restrict to augmentations sending $\lambda_i,\mu_i$ to particular values in $k^*$, then there is no bijection.
\end{rmk}

Given a framed cord $[c]$, its augmented value $\epsilon([c])$ will be abbreviated as $\epsilon(c)$. For example, applying an augmentation $\epsilon$ to (\ref{ConfDsr}) can be expressed as
$$\epsilon(c_{it} \cdot c_{tj}) = \epsilon(c_{it} \cdot m_t^{(t)}\cdot c_{tj}) + \epsilon(c_{it})\epsilon(c_{tj}).$$

Suppose $L$ is the closure of an $n$-strand braid and $\epsilon: \mathrm{Cord}(L)\rightarrow k$ is an augmentation. We define an $n\times n$ square matrix $R$, by $R_{ij} = \epsilon(\gamma_{ij})$. Namely,
$$
R = 
\begin{pmatrix}
\epsilon(\gamma_{11}) & \dotsb & \epsilon(\gamma_{1n}) \\
\vdots & \ddots &\vdots \\
\epsilon(\gamma_{n1}) & \dotsb & \epsilon(\gamma_{nn})
\end{pmatrix}.
$$

\begin{rmk}
Since Lemma \ref{StdCordsGen} asserts that standard cords generate $\mathrm{Cord}(L)$, it is natural to ask whether the matrix $R$ in turn determines the augmentation $\epsilon$. The answer is sometimes. For example when $K$ is a knot as in \cite[Theorem 4.16]{Ga2}, in the first two cases one can reconstruct the augmentation $\epsilon$ from the associated matrix $R$, while in the third case one needs to specify $\epsilon(\lambda)$ to recover $\epsilon$ from $R$.
\end{rmk}

\subsection{Augmentation representation}

We defined a matrix $R$ out of an augmentation $\epsilon: \mathrm{Cord}(L)\rightarrow k$. Let $R_j$ be the $j$-th column vector of $R$. Define a $k$-vector space
$$V_\epsilon := \textrm{Span}_k\{R_j\}_{1\leq j\leq n}.$$

We adopt the following convention to represent a column vector of size $n$. For any path $c_j$ from $x_0$ to $x_j$, define
$$\epsilon({p}_\alpha^{-1} \cdot c_j) = \big(\epsilon({p}_1^{-1} \cdot c_j),\dotsb, \epsilon(p_n^{-1}\cdot c_j)\big).$$
\begin{thm-defn}\label{MainConstruction}
Suppose $L$ is an oriented link equipped with its Seifert framing. Let $\epsilon: \mathrm{Cord}(L)\rightarrow k$ be an augmentation of its framed cord algebra. The following map defines a link group representation $\rho_\epsilon: \pi_L\rightarrow GL(V_\epsilon)$,
\begin{equation}\label{DefAugRep}
\rho_\epsilon(h) R_j: = \epsilon({p}_\alpha^{-1} \cdot h \cdot p_j), \textrm{  where } h\in \pi_1(X\setminus L, x_0).
\end{equation}
We call $(\rho_\epsilon,V_\epsilon)$ the \textit{augmentation representation} associated to $\epsilon$.
\end{thm-defn}

\begin{rmk}
The construction of $(\rho_\epsilon, V_\epsilon)$ depends on the choice of a braid, so as to define standard cords $\gamma_{ij}$ and the matrix $R$. However, the isomorphism class of the representation is independent from the braid, proven in Theorem \ref{MainThm}.
\end{rmk}

The action for a general loop $h$ might be difficult to compute, but the action of meridians on standard cords takes a simpler form. Check out the following lemma.

\begin{lem}\label{meridanRj}
For $1\leq t, i, j\leq n$, there are
\begin{equation}\label{MeridianAction}
\rho_\epsilon(m_t) R_j = R_j - \epsilon(\gamma_{tj})R_t, \qquad \rho_\epsilon(m_t^{-1}) R_j  =  R_j + \mu_{\{t\}}^{-1}\epsilon(\gamma_{tj})R_t.
\end{equation}
\end{lem}
\begin{proof}
We recall some notations and identities. For $1\leq t,i,j \leq n$, $p_i$ is a capping path from $x_0$ to $x_i$, $\gamma_{ij} = p_i^{-1} \cdot p_j$ is a standard cord from $x_i$ to $x_j$, $m_t$ is a meridian based at $x_0$, and $m_t^{(t)} = p_t^{-1}\cdot m_t \cdot p_t$ is a meridian cord based at $x_t$. 

To verify the first identity, by definition and interpolating $(p_t\cdot p_t^{-1})$, there is
$$
\rho_\epsilon(m_t) R_j := \epsilon(p_\alpha^{-1}\cdot m_t \cdot  p_j) 
= \epsilon(p_\alpha^{-1}\cdot (p_t\cdot p_t^{-1}) \cdot m_t \cdot (p_t\cdot p_t^{-1}) \cdot  p_j) 
= \epsilon(\gamma_{\alpha t} \cdot m_t^{(t)}\cdot \gamma_{tj}).
$$
By the skein relation, there is 
$$
\epsilon(\gamma_{\alpha t} \cdot m_t^{(t)}\cdot \gamma_{tj}) 
= \epsilon(\gamma_{\alpha t}\cdot \gamma_{tj}) - \epsilon(\gamma_{\alpha t})\epsilon(\gamma_{tj}) 
= R_j - \epsilon(\gamma_{tj})R_t.
$$
Combining these equations, we conclude that
$$\rho_\epsilon(m_t) R_j = R_j - \epsilon(\gamma_{tj})R_t.$$

In a similar fashion, we can derive for the second identity in the assertion: 
\begin{align*}
\rho_\epsilon(m_t^{-1}) R_j 
	&= \epsilon(p_\alpha^{-1}\cdot m_t^{-1} \cdot  p_j) = \epsilon(\gamma_{\alpha t} \cdot (m_t^{(t)})^{-1}\cdot \gamma_{tj}) \\
	&= \epsilon(\gamma_{\alpha t}\cdot (m_t^{(t)})\cdot(m_t^{(t)})^{-1}\cdot \gamma_{tj}) + \epsilon(\gamma_{\alpha t})\epsilon((m_t^{(t)})^{-1}\cdot \gamma_{tj}) \\
	& = R_j + \epsilon((m_t^{(t)})^{-1}\cdot \gamma_{tj})R_t.
\end{align*}
From the meridian relation, we know that $\epsilon((m_t^{(t)})^{-1}\cdot \gamma_{tj}) = \mu_{\{t\}}^{-1}\epsilon(\gamma_{tj})$. We conclude that
$$\rho_\epsilon(m_t^{-1}) R_j  = R_j + \mu_{\{t\}}^{-1}\epsilon(\gamma_{tj})R_t.$$
\end{proof}

Now we are ready to prove that the construction is indeed a representation. 

\begin{proof}[Proof of Theorem-Definition \ref{MainConstruction}]
We write $(\rho,V)$ instead of $(\rho_\epsilon, V_\epsilon)$ in this proof.

(1) To see that $\rho(h)$ is a closed linear map, it suffices to prove for meridian generators. Namely, for $h = m_t^{\pm 1}$ and any vector $R_j$, there is $\rho(m_t)R_j\in V$. By Lemma \ref{meridanRj}, we have
\begin{equation}\label{MeridianAction1}
\rho(m_t) R_j = R_j- \epsilon({p}_t^{-1}\cdot p_j) R_t, \quad \rho(m_t^{-1}) R_j  = R_j + \epsilon(p_t^{-1}\cdot m_t^{-1}\cdot p_j)R_t.
\end{equation}
For fixed $t$ and $j$, $\epsilon({p}_t^{-1}\cdot p_j)$ and $\epsilon(p_t^{-1}\cdot m_t^{-1}\cdot p_j)$ are constants in the field $k$. Therefore $\rho(m_t^{\pm 1})$ is closed.

(2) The identity is straightforward to check. Let $e$ be the identity loop based at $x_0$, then $\rho(e)R_j = R_j$ for all $j$ by definition.

(3) To verify the composition in the group action, one needs $\rho(h_1)\rho(h_2) = \rho(h_1\cdot h_2)$. Since each loop in $\pi_L$ can be expressed as a word of meridian generators, we can prove by inducting on the word length of $h_2$.

For the initial step, we take $h_1 \in \pi_L$, and $h_2 = m_t^{\pm 1}$. It suffices to prove that for any $R_j$, $\rho(h_1)\rho(m_t^{\pm 1})R_j = \rho(h_1\cdot m_t^{\pm 1})R_j$. For $m_t$, the left hand side is
\begin{align*}
\rho(h_1)\rho(m_t)R_j 
&=  \rho(h_1) \big( R_j- \epsilon({p}_t^{-1}\cdot p_j) R_t \big) \\
&= \epsilon(p_\alpha^{-1}\cdot h_1 \cdot p_j) -\epsilon({p}_t^{-1}\cdot p_j)\epsilon(p_{\alpha}^{-1}\cdot h_1 \cdot p_t). 
\end{align*}
The first equation follows (\ref{MeridianAction}) and the second equation is by definition (\ref{DefAugRep}). Continuing with the right hand side, there is
\begin{align*}
\rho(h_1\cdot m_t)R_j 
&= \epsilon(p_\alpha^{-1}\cdot h_1\cdot m_t \cdot p_j) \\
&= \epsilon(p_\alpha^{-1}\cdot h_1\cdot p_j) - \epsilon(p_\alpha^{-1}\cdot h_1 \cdot p_t) \epsilon(p_t^{-1}\cdot p_j).
\end{align*}
The first equation is by definition (\ref{DefAugRep}) and the second equation is by skein relations. Comparing these equations, we see that $\rho(h_1)\rho(m_t)R_j = \rho(h_1\cdot m_t)R_j$ for any $j= 1,\dotsb, n$.

The proof for the case $h_2 = m_t^{-1}$ is similar. Namely,
$$\rho(h_1)\rho(m_t^{-1})R_j = \epsilon(p_{\alpha}^{-1}\cdot h_1\cdot p_j) + \epsilon( p_t^{-1}\cdot m_t^{-1}\cdot p_j)\epsilon (p_{\alpha}^{-1}\cdot h_1 \cdot p_t) = \rho(h_1\cdot m_t^{-1})R_j.$$

Proceeding to the induction step, we assume that $\rho(h_1)\rho(h_2') = \rho(h_1\cdot h_2')$ for any $h_1\in \pi_L$, and any $h_2'\in \pi_L$ that can be written as a word of meridian generators with length less or equal to $s$. Assume $h_2 = h_2' m_t^{\pm 1}$ has word length less or equal to $s+1$, then
$$
\rho(h_1)\rho(h_2' m_t^{\pm 1})=\rho(h_1)\rho(h_2')\rho(m_t^{\pm 1}) = \rho(h_1h_2')\rho(m_t^{\pm 1}) = \rho(h_1h_2'm_t^{\pm 1}).
$$
Each of the three equalities follows from the induction hypothesis. Therefore we complete the induction and show that $\rho(h_1)\rho(h_2) = \rho(h_1\cdot h_2)$. 

(4) We prove that $\rho(h)$ is a well-defined linear map, namely if $I \subset \{1,\dotsb, n\}$ is a subset and $\sum_{i\in I} a_i R_i=0$ for some constants $a_i$, then $\sum_{i\in I} a_i \rho(h)R_i=0$ as well. If $h = m_t$, by Lemma \ref{meridanRj}, there is
$$\sum_{i\in I} a_i \rho(h)R_i = \sum_{i\in I} a_i R_i - \sum_{i\in I} a_i \epsilon(\gamma_{ti}) R_t.$$
The first summand is zero by hypothesis. In the second summand, $\sum_{i\in I} a_i \epsilon(\gamma_{ti})$ equals to the $t$-th row of $\sum_{i\in I} a_i R_i=0$, which is also zero. The argument for $h =m_t^{-1}$ is similar. Continuing by an induction on the word length of $h$ in terms of meridian generators, we prove the well-definedness.

\end{proof}

Since an isotopy of a framed link can be extended to an isotopy of the ambient manifold, and every framed cord finds its counterpart following the isotopy, we see that $\mathrm{Cord}(L)$ is an invariant of framed links. An augmentation $\epsilon:\mathrm{Cord}(L)\rightarrow k $, algebraically defined as an algebra morphism, is also invariant under isotopies.

However, the construction of $(\rho_\epsilon, V_\epsilon)$ requires a braid representative. In the following theorem, we show the representation is independent from the choice of the braid. 

\begin{thm}\label{MainThm}
Suppose $L$ is an oriented link equipped with its Seifert framing. Let $\epsilon: \mathrm{Cord}(L)\rightarrow k$ be an augmentation of its framed cord algebra. Up to isomorphism, the augmentation representation $(\rho_\epsilon,V_\epsilon)$ is well-defined for the augmentation $\epsilon$. In particular, it does not depend on the choice of the braid in the construction.
\end{thm}
\begin{proof}

Suppose $B$ is a braid whose closure is $L$. Let $(\rho_\epsilon,V_\epsilon)$ be the augmentation representation of $\epsilon$ constructed with respect to $B$ as in Theorem-Definition \ref{MainConstruction}. Suppose $\tilde{B}$ is another braid whose closure is also $L$, and let $(\tilde{\rho}_\epsilon,\tilde{V}_\epsilon)$ be the representation constructed with respect to $\tilde{B}$. We shall prove that $(\rho_\epsilon,V_\epsilon)$ and $(\tilde{\rho}_\epsilon,\tilde{V}_\epsilon)$ are isomorphic as $\pi_L$-representations.

In the rest of the proof, we drop the subscript and write $(\rho,V), (\tilde{\rho},\tilde{V})$ for $(\rho_\epsilon,V_\epsilon), (\tilde{\rho}_\epsilon,\tilde{V}_\epsilon)$

To prove the isomorphism, we will construct a linear map $T: V\rightarrow \tilde{V}$ which intertwines with link group actions. Namely, for any $h \in \pi_L$, there should be a commutative diagram,
\begin{center}
\begin{tikzpicture}
  \node (A){$V$};
  \node (B)[right of=A, node distance=2cm]{$\tilde{V}$};
  \node (C)[below of=A, node distance=2cm]{$V$};
  \node (D)[right of=C, node distance=2cm]{$\tilde{V}$};
  \draw[->] (A) to node [] {$T$} (B);
  \draw[->] (A) to node [swap] {$\rho(h)$} (C);
  \draw[->] (B) to node []{$\tilde\rho(h)$}(D);
  \draw[->] (C) to node []{$T$}(D);
\end{tikzpicture}
\end{center}
which can be restated as $\tilde{\rho}(h)\circ T = T\circ \rho(h)$. Since the link group is generated by a set of meridian generators, it suffices to check the commutativity for $h$ being a meridian generator.

By Markov's theorem, two braids have isotopic closures if and only if they are related by a sequence of equivalent relations: (1) they are equivalent braids, (2) they are conjugate braids, (3) one braid is a positive/negative stabilization of the other braid. It is clear that the construction of the augmentation representation is well-defined within an equivalence class of braids. We will verify the well-definedness in other cases.

\medskip
{\textbf{Conjugations.}} 
Suppose $B$ is an $n$-strand braid. Let $\sigma_s$, $1\leq s\leq n-1$ be the positive half twist in the braid group $Br_n$. Let $\tilde B = \sigma_s B\sigma_s^{-1}$ be a conjugation. 

\begin{figure}[h]
	\centering
	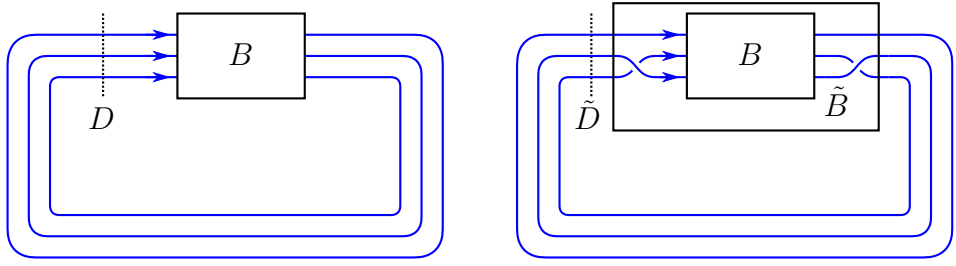
	\caption[]{Conjugation.}
	\label{Fig:Conjugation}
\end{figure}

We write $\gamma_{ij}$ for the standard cords in the configuration disk $D$, and $\tilde{\gamma}_{ij}$ (abbreviated from a more rigorous notation $\tilde{\gamma}_{\tilde{i}\tilde{j}}$) for the standard cords in the configuration disk $\tilde{D}$ that is used to define $(\tilde{\rho},\tilde{V})$. These standard cords are related as elements in $\mathrm{Cord}(L)$, expressed as follows: for $i,j \neq s,s+1$,
\begin{align}\label{conjugationcords}
\tilde{\gamma}_{ij}  & = \gamma_{ij}, \nonumber\\
\tilde\gamma_{i,s+1} &= \gamma_{is}, & \tilde{\gamma}_{is} &= \gamma_{is} \cdot m_s^{(s)} \cdot \gamma_{s,s+1}, \nonumber\\
\tilde\gamma_{s+1,j} &= \gamma_{sj}, & \tilde{\gamma}_{sj} &= \gamma_{s+1,s} \cdot (m_s^{(s)})^{-1} \cdot \gamma_{sj},\\
\tilde{\gamma}_{s,s} &= \gamma_{s+1,s+1}, & \tilde{\gamma}_{s,s+1} &= \gamma_{s+1,s}\cdot (m_s^{(s)})^{-1}, \nonumber\\
\tilde{\gamma}_{s+1,s+1} &= \gamma_{s,s}, & \tilde{\gamma}_{s+1,s} &= m_{s}^{(s)}\cdot \gamma_{s,s+1}.\nonumber
\end{align}
We related the two sets of meridian generators $\{m_t\}_{1\leq t\leq n}$, $\{\tilde{m}_t\}_{1\leq t\leq n}$ in a similar way:
\begin{align*}
&\tilde{m}_{s+1} = m_s,\quad \tilde{m}_s = m_s\cdot m_{s+1}\cdot m_s^{-1},\\
&\tilde{m}_t = m_t \quad \textrm{ for } \; t\neq s,s+1.
\end{align*}
Again, $\tilde{m}_{t}$ is an abbreviation of $\tilde{m}_{\tilde{t}}$.

We define matrices $R,\tilde{R}$ by $R_{ij} = \epsilon(\gamma_{ij}), \tilde{R}_{ij} = \epsilon(\tilde\gamma_{ij})$. Following the construction of the augmentation representation, we define
$$V := \textrm{Span}_k\{R_j\}_{1\leq j\leq n}, \quad \tilde{V} := \textrm{Span}_k\{\tilde{R}_j\}_{1\leq j\leq n}.$$
We first verify that $V$ and $\tilde{V}$ have the same rank. There exist $n\times n$ invertible matrices $M_L$ and $M_R$ such that $\tilde{R} = M_L R M_R$, where
$$
M_L =
\begin{pmatrix}
I_{s-1} & & \\
& M_L' & \\
& & I_{n-s-1}
\end{pmatrix},
\quad
M_R =
\begin{pmatrix}
I_{s-1} & & \\
& M_R' & \\
& & I_{n-s-1}
\end{pmatrix},
$$
and $M_L'$, $M_R'$ are the following $2\times 2$ matrices
$$
M_L'  = 
\begin{pmatrix}
\epsilon(\gamma_{s+1,s}\cdot(m_s^{(s)})^{-1}) & 1 \\
1 & 0
\end{pmatrix},
\quad
M_R'  = 
\begin{pmatrix}
-\epsilon(\gamma_{s,s+1}) & 1 \\
1 & 0
\end{pmatrix}.
$$
Observe that $M_L,M_R$ only affect rows and columns that are indexed by $s,s+1$. If neither neither of $i$ or $j$ belongs to $\{s,s+1\}$, then
$$(M_L R M_R)_{ij}  =R_{ij} = \epsilon(\gamma_{ij}) = \epsilon(\tilde{\gamma}_{ij}) = \tilde{R}_{ij}.$$
If only one of the row or column index is in $\{s,s+1\}$, then 
\begin{align*}
    (M_L R M_R)_{sj} = (M_LR)_{sj} &= \epsilon(\gamma_{s+1, s}\cdot (m_s^{(s)})^{-1}) \epsilon(\gamma_{sj}) + \epsilon(\gamma_{s+1,j}) \\
    &= \epsilon (\gamma_{s+1,s} \cdot (m_s^{(s)})^{-1} \cdot \gamma_{sj}) = \epsilon(\tilde{\gamma}_{sj}) = \tilde{R}_{sj},\\
    (M_L R M_R)_{s+1, j} = (M_LR)_{s+1,j} &= \epsilon(\gamma_{sj}) = \epsilon(\tilde \gamma _{s+1,j}) = \tilde{R}_{s+1,j},\\
    (M_L R M_R)_{is} = (RM_R)_{is} &= -\epsilon(\gamma_{is})\epsilon(\gamma_{s, s+1}) + \epsilon(\gamma_{i,s+1}) \\
    &= \epsilon (\gamma_{is} \cdot m_s^{(s)} \cdot \gamma_{s,s+1}) = \epsilon(\tilde{\gamma}_{is}) = \tilde{R}_{is},\\
    (M_L R M_R)_{i, s+1} = (RM_R)_{i, s+1} &= \epsilon(\gamma_{is}) = \epsilon(\tilde \gamma_{i,s+1}) = \tilde{R}_{i,s+1}.
\end{align*}
Finally we verify for the $2\times 2$ sub-matricies whose indices are contained in $[s,s+1]^2$. Let $A_{[s,s+1]^2}$ be the $2\times 2$ sub-matrix of $A$ consisting entries with both row and column indices in $\{s,s+1\}$. Let $\epsilon(\gamma_{s+1,s}) = a,\epsilon(\gamma_{s,s+1}) = b$, then
\begin{align*}
(M_L R M_R)_{[s,s+1]^2} &= M_L' R _{[s,s+1]^2} M_R' \\
&= 
\begin{pmatrix}
a\mu_{\{s\}}^{-1} & 1 \\
1 & 
\end{pmatrix}
\begin{pmatrix}
1-\mu_{\{s\}} & b \\
a & 1- \mu_{\{s+1\}}
\end{pmatrix}
\begin{pmatrix}
-b & 1 \\
1 & 
\end{pmatrix} \\
&=
\begin{pmatrix}
1- \mu_{\{s+1\}} & a\mu_{\{s\}}^{-1} \\
b\mu_{\{s\}}& 1-\mu_{\{s\}}
\end{pmatrix}= \tilde{R}_{[s,s+1]^2},
\end{align*}
where the last equality is because of $\epsilon(\tilde{\gamma}_{s,s+1}) = \epsilon(\gamma_{s+1,s}\cdot (m_s^{(s)})^{-1})= \epsilon(\gamma_{s+1,s})\mu_{\{s\}}^{-1} =  a\mu_{\{s\}}^{-1}$, and similarly $\epsilon(\tilde{\gamma}_{s,s+1}) = b\mu_{\{s\}}$. Therefore $\tilde{R} = M_L R M_R$. Because $M_L$ and $M_R$ are invertible, $V$ and $\tilde{V}$ are isomorphic as vector spaces.

\smallskip
Define a linear map $T: V\rightarrow \tilde{V}$ over the spanning vectors:
\begin{align*}
T(R_s) &= \tilde{R}_{s+1}, \quad T(R_{s+1}) = \tilde{R}_s + \epsilon(\gamma_{s,s+1})\tilde{R}_{s+1},\\
T(R_j) &= \tilde{R}_j \;\; \textrm{for} \;\; j\neq s,s+1.
\end{align*}
The inverse map $T^{-1}$ is thusly determined:
\begin{align*}
T^{-1}(\tilde{R}_s) &= {R}_{s+1} - \epsilon(\gamma_{s,s+1})R_s, \quad T^{-1}(\tilde{R}_{s+1}) = {R}_s,\\
T^{-1}(\tilde{R}_j) &= {R}_j \;\; \textrm{for} \;\; j\neq s,s+1.
\end{align*}
In terms of the spanning vectors, $T$ and $T^{-1}$ can be expressed as $n\times n$ matrices:
\begin{equation}\label{CongT}
T = 
\begin{pmatrix}
I_{s-1} &&&\\
& 0 & 1 & \\
& 1 & \epsilon(\gamma_{s,s+1}) & \\
& & & I_{n-s-1}
\end{pmatrix},
\quad 
T^{-1}
=
\begin{pmatrix}
I_{s-1} &&&\\
& -\epsilon(\gamma_{s,s+1}) & 1 & \\
& 1 & 0 & \\
& & & I_{n-s-1}
\end{pmatrix}.
\end{equation}
Since the row operations given by $M_L$ do not change the linear dependence relations among the spanning vectors of $V$, and the matrix $M_R$ for column operations is consistent with $T^{-1}$, the map $T$ descents to a linear transform $V\rightarrow \tilde {V}$.

\smallskip
Finally we show that $\tilde{\rho}(\gamma)\circ T = T\circ \rho(\gamma)$ for any $\gamma = m_t$, $t = 1, \dotsb, n$. We introduce some notations. Following the calculations in Lemma \ref{meridanRj}, $\rho(m_t)$ and $\tilde{\rho}(\tilde{m}_t)$ can be written as $n\times n$ matrices with respect to their spanning vectors:
\begin{equation}\label{CongM}
M_t: =\rho(m_t)= I_n - \sum_{j=1}^n\epsilon(\gamma_{tj})E_{tj}, \quad \tilde{M}_t :=\tilde\rho(\tilde m_t)= I_n - \sum_{j=1}^n\epsilon(\tilde{\gamma}_{tj})E_{tj},
\end{equation}
where $E_{ij}$ is the matrix having $1$ at the $(i,j)$-th entry and $0$ elsewhere.

We use Kroneker delta $\delta^i_j$ to denote a number which equals to $1$ when $i=j$ and $0$ otherwise. We do not use Einstein's convention for contractions, and summations will be indicated.
 
Introducing new matrices $N_t := I_n - M_t$ and $\tilde{N}_t := I_n - \tilde{M}_t$, then
\begin{align}\label{SubstitutionMN}
\begin{split}
(M_t)_{ij} &= \delta^i_j - \epsilon(\gamma_{tj})\delta^t_i,\quad (N_t)_{ij} = \epsilon(\gamma_{tj})\delta^t_i, \\
(\tilde{M}_t)_{ij} &= \delta^i_j - \epsilon(\tilde\gamma_{tj})\delta^t_i, \quad (\tilde{N}_t)_{ij} = \epsilon(\tilde\gamma_{tj})\delta^t_i.
\end{split}
\end{align}

We are ready to proceed into details. Recall our goal is to verify
\begin{equation}\label{MainThmCongRepCom}
\tilde{\rho}(m_t) \circ T = T \circ \rho(m_t).
\end{equation}
Note that on the left hand side, it is $\tilde{\rho}(m_t)$ instead of $\tilde{\rho}(\tilde{m}_t)$. We consider the following three cases based on the value of $t$: (A) $t\neq s,s+1$, (B) $t=s$ and (C) $t=s+1$.

\smallskip 
\underline{(A) $t\neq s,s+1$.}

We have $m_t = \tilde{m}_t$, and equation (\ref{MainThmCongRepCom}) becomes $\tilde{M}_t T = T M_t$. Invoking the relations $N_t = I_n - M_t$, $\tilde{N}_t = I_n - \tilde{M}_t$, it suffices to check 
$$\tilde{N}_t T = T{N}_t.$$ 

The right hand side is $(TN_t)_{ij} = \sum_eT_{ie}(N_t)_{ej} = \sum_eT_{ie}\epsilon(\gamma_{tj})\delta^t_e = T_{it}\epsilon(\gamma_{tj})$. Since $t\neq s,s+1$, there is $T_{it} = \delta^i_t$. Hence if $i\neq t$, then $(TN_t)_{ij}=0$; if $i = t$, then $(TN_t)_{ij}=\epsilon(\gamma_{ij})$.

The left hand side if $(\tilde{N}_t T)_{ij} = \sum_d (\tilde{N}_t)_{id} T_{dj} = \sum_d \epsilon(\tilde{\gamma}_{td})\delta^t_i T_{dj}$. If $i\neq t$, then $(\tilde{N}_t T)_{ij} =0$. If $i=t$, then $(\tilde{N}_t T)_{ij} = \sum_d \epsilon(\tilde{\gamma}_{id}) T_{dj}$.

\begin{itemize}

\item If $j\neq s,s+1$, then $T_{dj} = \delta^d_j$, and hence $(\tilde{N}_t T)_{tj} = \epsilon(\tilde{\gamma}_{ij}) = \epsilon({\gamma}_{ij})$.

\item If $j=s$, then $(\tilde{N}_t T)_{ts} = \epsilon(\tilde{\gamma}_{t,s+1}) = \epsilon({\gamma}_{ts})$.

\item If $j=s+1$, then 
$$\qquad (\tilde{N}_t T)_{t,s+1} = \epsilon(\tilde{\gamma}_{ts}) + \epsilon(\tilde{\gamma}_{t,s+1}) \epsilon(\gamma_{s,s+1}) 
	= \epsilon(\gamma_{ts} \cdot m_s^{(s)} \cdot \gamma_{s,s+1}) + \epsilon({\gamma}_{ts})\epsilon(\gamma_{s,s+1}) 
	= \epsilon(\gamma_{t,s+1}).$$

\end{itemize}

\smallskip 
\underline{(B) $t = s$.}

We have $m_s = \tilde{m}_{s+1}$ and equation (\ref{MainThmCongRepCom}) becomes $\tilde{M}_{s+1} T = T{M}_s$. It suffices to verify 
$$\tilde{N}_{s+1} T = T{N}_s.$$

\begin{itemize}
\item If $i \neq s+1$, then $(\tilde{N}_{s+1} T)_{ij} =  0 = (T{N}_s)_{ij}$, because
\begin{align*}
(\tilde{N}_{s+1} T)_{ij}&= \sum_d(\tilde{N}_{s+1})_{id} T_{dj} =  \sum_d\epsilon(\tilde{\gamma}_{s+1,j})\delta^{s+1}_i T_{dj} = 0, \\
(T{N}_{s})_{ij}&= \sum_e T_{ie} ({N}_{s})_{ej} = \sum_e T_{ie} \epsilon({\gamma}_{sj})\delta^{s}_e = T_{is} \epsilon({\gamma}_{sj}) = 0.
\end{align*}

\item If $i=s+1$. 

	- When $j\neq s,s+1$, then $(\tilde{N}_{s+1}T)_{s+1, j} =  \epsilon(\tilde{\gamma}_{s+1,j}) = \epsilon({\gamma}_{sj}) = (T{N}_s)_{s+1, j}$.

	- When $j=s$, then $(\tilde{N}_{s+1}T)_{s+1, s} = \epsilon(\tilde{\gamma}_{s+1,s+1}) \stackrel{(\ref{conjugationcords})}{=} \epsilon(\gamma_{s,s}) = (T{N}_{s})_{s+1,s} $.

\smallskip

	- When $j=s+1$, then
$$
\qquad (\tilde{N}_{s+1}T)_{s+1, s+1} = \epsilon(\tilde\gamma_{s+1,s}) + \epsilon(\tilde{\gamma}_{s+1,s+1})\epsilon(\gamma_{s,s+1}),\quad
(T{N}_s)_{s+1,s+1} = \epsilon(\gamma_{s,s+1}).
$$

\noindent Note $\epsilon(\tilde{\gamma}_{s+1,s+1}) = \epsilon(\gamma_{s,s}) = 1- \mu_{\{s\}}$ and $\epsilon(\tilde{\gamma}_{s+1,s}) = \epsilon(m_{s}^{(s)}\cdot \gamma_{s,s+1})= \mu_{\{s\}}\epsilon(\gamma_{s,s+1}).$
Therefore,
$$\quad \quad \epsilon(\tilde\gamma_{s+1,s}) + \epsilon(\tilde{\gamma}_{s+1,s+1})\epsilon(\gamma_{s,s+1}) = \mu_{\{s\}}\epsilon(\gamma_{s,s+1}) + (1- \mu_{\{s\}}) \epsilon(\gamma_{s,s+1}) = \epsilon (\gamma_{s,s+1}).$$

\end{itemize}

\smallskip 
\underline{(C) $t= s+1$.}

We have $m_{s+1} = \tilde{m}_{s+1}^{-1}\cdot \tilde{m}_s\cdot \tilde{m}_{s+1}$. This time equation (\ref{MainThmCongRepCom}) becomes 
$$\tilde{M}_{s+1}^{-1}\tilde{M}_s \tilde{M}_{s+1} T = TM_{s+1},$$ 
and it further simplifies to
\begin{equation}\label{congcase3eq}
\tilde{M}_{s+1}^{-1} \tilde{N}_s \tilde{M}_{s+1} T = TN_{s+1}.
\end{equation}

The right hand side of (\ref{congcase3eq}) at the $(i,j)$-entry is 
$$(TN_{s+1})_{ij} = \sum_{d}T_{id}(N_{s+1})_{dj} = \sum_{d}T_{id}\,\epsilon(\gamma_{s+1,j})\delta^d_{s+1} = T_{i,s+1}\epsilon(\gamma_{s+1,j}).$$

\noindent Since row $s+1$ of $T$ is mostly zero except for $i = s,s+1$, the right hand side of (\ref{congcase3eq}) is
\begin{equation}\label{congCaseCRHS}
\begin{cases}
(TN_{s+1})_{ij} =0,  \qquad i \neq s,s+1, \\
(TN_{s+1})_{sj} = \epsilon(\gamma_{s+1,j}), \\
 (TN_{s+1})_{s+1, j} = \epsilon({\gamma}_{s,s+1})
 \epsilon(\gamma_{s+1,j}).
\end{cases}
\end{equation}

For the left hand side of (\ref{congcase3eq}). We claim the following expansion

\begin{equation}\label{CongCaseCLHSexpansion}
(\tilde{M}_{s+1}^{-1}\tilde{N}_s \tilde{M}_{s+1} T)_{ij} = (\tilde{M}_{s+1}^{-1})_{is} \cdot \sum_f \epsilon(\tilde{\gamma}_{s,s+1} \cdot \tilde{m}_{s+1}^{(s+1)} \cdot \tilde{\gamma}_{s+1,f})T_{fj}.
\end{equation}
Proof of claim (\ref{CongCaseCLHSexpansion}): Using (\ref{SubstitutionMN}), we can express $(\tilde{M}_{s+1}^{-1}\tilde{N}_s \tilde{M}_{s+1} T)_{ij}$ as
\begin{align*}
 \sum_{d,e,f} (\tilde{M}_{s+1}^{-1})_{id}(\tilde{N}_s)_{de}(\tilde{M}_{s+1})_{ef}T_{fj} 
&= \sum_{d,e,f} (\tilde{M}_{s+1}^{-1})_{id}\,\epsilon(\tilde{\gamma}_{se})\delta^s_d(\tilde{M}_{s+1})_{ef}T_{fj} \\
&=(\tilde{M}_{s+1}^{-1})_{is} \cdot \sum_{e,f} \epsilon(\tilde{\gamma}_{se})(\tilde{M}_{s+1})_{ef}T_{fj}.
\end{align*}
The term $\epsilon(\tilde{\gamma}_{se})(\tilde{M}_{s+1})_{ef}$ can be rewritten as
\begin{align*}
\epsilon(\tilde{\gamma}_{se})(\tilde{M}_{s+1})_{ef} &= \epsilon(\tilde{\gamma}_{se}) (\delta^e_f - \epsilon(\tilde\gamma_{s+1,f})\delta^{s+1}_e) \\
&= \epsilon(\tilde\gamma_{sf}) - \epsilon(\tilde\gamma_{s+1,f})\epsilon(\tilde\gamma_{s,s+1}) = \epsilon(\tilde{\gamma}_{s,s+1} \cdot \tilde{m}_{s+1}^{(s+1)} \cdot \tilde{\gamma}_{s+1,f}).
\end{align*}
Combining these terms and we verify the claim.

\smallskip
Observe that the expansion in \eqref{CongCaseCLHSexpansion} is the product of two terms, one depends only on $i$ and the other depends only on $j$. Consider the term with $i$, $(\tilde{M}_{s+1}^{-1})_{is}$. Given the expression of $\tilde{M}_t$, it is straightforward to compute its inverse
$$\tilde{M}_t^{-1} = I_n + \sum_{j=1}^{n} \tilde{\mu}_{\{t\}}^{-1}\epsilon(\tilde{\gamma}_{tj})E_{tj}, 
\qquad (\tilde{M}_t^{-1})_{ij} = \delta^i_j + \tilde{\mu}_{\{t\}}^{-1} \epsilon(\tilde{\gamma}_{tj}) \delta^t_i.$$
Let $t= s+1$ and we get
$$(\tilde{M}_{s+1}^{-1})_{is} = \delta^i_s + \tilde{\mu}_{\{s+1\}}^{-1} \epsilon(\tilde{\gamma}_{s+1, s}) \delta^{s+1}_i.$$

\begin{itemize}
\item If $i\neq s,s+1$, $(\tilde{M}_{s+1}^{-1})_{is} = 0$.

\item If $i =s$, then $(\tilde{M}_{s+1}^{-1})_{is} = 1$.

\item If $i=s+1$, then $(\tilde{M}_{s+1}^{-1})_{is} = \tilde{\mu}_{\{s+1\}}^{-1}\epsilon(\tilde{\gamma}_{s+1, s})=\epsilon(\gamma_{s,s+1})$, because
$$ \tilde{\mu}_{\{s+1\}}^{-1}\epsilon(\tilde{\gamma}_{s+1, s})= \epsilon((\tilde{m}_{s+1}^{(s+1)})^{-1}\cdot \tilde{\gamma}_{s+1,s}) = \epsilon((m_s^{s})^{-1}\cdot m_s^{s}\cdot \gamma_{s,s+1}) = \epsilon(\gamma_{s,s+1}).$$

\end{itemize}

Consider the term involving $j$. We first present the answer, which is
\begin{equation}\label{congCaseCtermj}
\sum_f \epsilon(\tilde{\gamma}_{s,s+1} \cdot \tilde{m}_{s+1}^{(s+1)} \cdot \tilde{\gamma}_{s+1,f})T_{fj} = \epsilon(\gamma_{s+1, j}).
\end{equation}
Multiplying it with the term involving $i$, the result matches (\ref{congCaseCRHS}), proving (\ref{congcase3eq}).

\smallskip
We verify (\ref{congCaseCtermj}) in the following cases.

\begin{itemize}
\item If $j\neq s,s+1$, then $T_{fj} = \delta^f_j$, and $\sum_f \epsilon(\tilde{\gamma}_{s,s+1} \cdot \tilde{m}_{s+1}^{(s+1)} \cdot \tilde{\gamma}_{s+1,f})T_{fj} $ becomes
$$
\epsilon(\tilde{\gamma}_{s,s+1} \cdot \tilde{m}_{s+1}^{(s+1)} \cdot \tilde{\gamma}_{s+1,j}) 
= \epsilon(\gamma_{s+1,s}\cdot (m_s^{(s)})^{-1}\cdot m_s^{(s)}\cdot \gamma_{sj}) = \epsilon(\gamma_{s+1,j}).
$$

\item If $j=s$, then $T_{fs} = \delta^{f}_{s+1}$, and $\sum_f \epsilon(\tilde{\gamma}_{s,s+1} \cdot \tilde{m}_{s+1}^{(s+1)} \cdot \tilde{\gamma}_{s+1,f})T_{fs}$ becomes
$$
\epsilon(\tilde{\gamma}_{s,s+1} \cdot \tilde{m}_{s+1}^{(s+1)} \cdot \tilde{\gamma}_{s+1,s+1}) 
= \epsilon(\gamma_{s+1,s}\cdot (m_s^{(s)})^{-1}\cdot m_s^{(s)}\cdot \gamma_{ss}) = \epsilon(\gamma_{s+1,s}).
$$

\item If $j=s+1$, then $T_{f,s+1} = \delta^{f}_{s} + \epsilon(\gamma_{s,s+1})\delta^{f}_{s+1}$, and 
\begin{align*}
&\sum_f \epsilon(\tilde{\gamma}_{s,s+1} \cdot \tilde{m}_{s+1}^{(s+1)} \cdot \tilde{\gamma}_{s+1,f})T_{f,s+1} \\
=& \epsilon(\tilde{\gamma}_{s,s+1} \cdot \tilde{m}_{s+1}^{(s+1)} \cdot \tilde{\gamma}_{s+1,s}) + \epsilon({\gamma}_{s,s+1})\epsilon(\tilde{\gamma}_{s,s+1} \cdot \tilde{m}_{s+1}^{(s+1)} \cdot \tilde{\gamma}_{s+1,s+1}) \\
=& \epsilon(\gamma_{s+1,s}\cdot (m_s^{(s)})^{-1}\cdot m_s^{(s)}\cdot m_s^{(s)} \cdot \gamma_{s,s+1}) +  \epsilon({\gamma}_{s,s+1}) \epsilon(\gamma_{s+1,s}) \\
=& \epsilon(\gamma_{s+1,s}\cdot m_s^{(s)} \cdot \gamma_{s,s+1}) +  \epsilon(\gamma_{s+1,s})\epsilon({\gamma}_{s,s+1}) \\
=& \epsilon(\gamma_{s+1,s+1}).
\end{align*}
\end{itemize}

\noindent The second equality follows from (\ref{conjugationcords}), and the last equality follows from skein relations.

To summarize, we have proven (\ref{MainThmCongRepCom}), and the augmentation representations before and after a conjugation are isomorphic as representations.

\medskip
{\textbf{Positive/Negative stabilizations.}} 
A stabilization of a braid closure is given by adding one strand and performing a positive/negative half twist with the outmost strand. The braid closure of a stabilized braid is equivalence to the closure of the original braid by a Reidemeister I move. Therefore stabilizations do not change the isotopy class of a link.

 Let $\iota: Br_n\rightarrow Br_{n+1}$ be the natural inclusion by adding a trivial strand labelled by $n+1$. Suppose $B$ is an $n$-strand braid, its positive stabilization is $\sigma_{n}\iota(B)$ and its negative stabilization is $\sigma_{n}^{-1}\iota(B)$.

\begin{figure}[h]
	\centering
	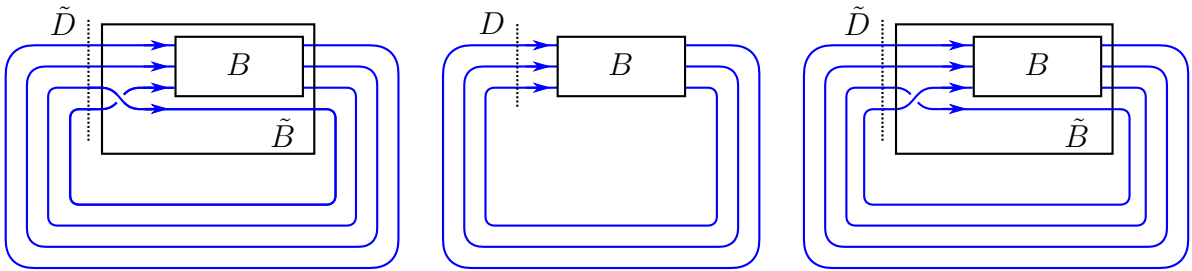
	\caption[]{Stabilizations: positive (left) and negative (right).}
	\label{Fig:Stabilizations}
\end{figure}

Let $B$ be an $n$-strand braid whose closure is $L$, let $\tilde{B} = \sigma_{n}\iota(B)$ be its positive stabilization. We adopt similar notations as before. Coefficients involving $\mu$ are related by
\begin{align}
\begin{split}
&\tilde\mu_{\{n\}} = \tilde{\mu}_{\{n+1\}} = \mu_{\{n\}},\\
&\tilde\mu_{\{t\}} = \mu_{\{t\}} \quad \textrm{ for } \; t\neq n, n+1.
\end{split}
\end{align}
 
Let $\gamma_{ij}, 1\leq i,j\leq n$ and $\tilde{\gamma}_{ij}, 1\leq i,j\leq n+1$ be the standard cords for $B$ and $\tilde{B}$. Note that $B$ and $\tilde{B}$ don't have the same number of strands, and the indices for the two sets of framed cords are different. For $i,j\neq n,n+1$,
\begin{align}
\tilde\gamma_{ij} &= \gamma_{ij}, &\nonumber\\
\tilde\gamma_{i,n+1} &= \gamma_{in}, & \tilde\gamma_{i,n} &= \gamma_{i,n}\cdot m_n^{(n)},\nonumber\\
\tilde\gamma_{n+1, j} &= \gamma_{nj}, & \tilde\gamma_{n,j} &= (m_n^{(n)})^{-1} \cdot \gamma_{nj}, \\
\tilde{\gamma}_{nn} &= \gamma_{nn}, & \tilde\gamma_{n,n+1}  &= (m_n^{(n)})^{-1}, \nonumber\\
\tilde\gamma_{n+1,n}  &= m_n^{(n)}, & \tilde{\gamma}_{n+1,n+1} &= \gamma_{nn}. \nonumber
\end{align}

\noindent The meridian generators are related by:
\begin{align}\label{meridiansPoSt}
\begin{split}
&\tilde{m}_{n+1} = \tilde{m}_n = m_n, \\
&\tilde{m}_t = m_t \quad \textrm{ for } \; t\neq n,n+1.
\end{split}
\end{align}

We define matrices $R,\tilde{R}$ by $R_{ij} := \epsilon(\gamma_{ij}), \tilde{R}_{ij} := \epsilon(\tilde\gamma_{ij})$, and
$$V := \textrm{Span}_k\{R_j\}_{1\leq j\leq n}, \quad \tilde{V} := \textrm{Span}_k\{\tilde{R}_j\}_{1\leq j\leq n+1}.$$

We first show that $V$ and $\tilde{V}$ have the same rank. Recall the identities $\epsilon(m_{n}^{(n)}) = \mu_{\{n\}}\epsilon(\gamma_{nn})$ and $\epsilon((m_{n}^{(n)})^{-1}) = \mu_{\{n\}}^{-1}\epsilon(\gamma_{nn})$. Hence,
\begin{equation}
\tilde{R} = 
\begin{pmatrix}
\epsilon(\gamma_{11}) & \dotsb & \epsilon(\gamma_{1,n-1}) & \mu_{\{n\}}\epsilon(\gamma_{1n}) & \epsilon(\gamma_{1n}) \\
\vdots & \ddots & \vdots & \vdots &\vdots \\
\epsilon(\gamma_{n-1,1}) & \dotsb &\epsilon(\gamma_{n-1,n-1}) & \mu_{\{n\}}\epsilon(\gamma_{n-1,n}) &  \epsilon(\gamma_{n-1,n})\\
\mu^{-1}_{\{n\}}\epsilon(\gamma_{n1}) & \dotsb & \mu^{-1}_{\{n\}}\epsilon(\gamma_{n, n-1}) & \epsilon(\gamma_{nn}) & \mu^{-1}_{\{n\}}\epsilon(\gamma_{nn})\\
\epsilon(\gamma_{n1}) & \dotsb & \epsilon(\gamma_{n, n-1}) & \mu_{\{n\}}\epsilon(\gamma_{nn}) & \epsilon(\gamma_{nn})
\end{pmatrix}.
\end{equation}
Set $M_L = I_{n+1} - \mu_{\{n\}}E_{n+1,n}$ and $M_R = I_{n+1} - \mu_{\{n\}}^{-1}E_{n,n+1}$, then
$$
M_L \tilde{R} M_R =
\begin{pmatrix}
R & 0 \\
0 & 0
\end{pmatrix}.
$$
Because $M_L$ and $M_R$ are invertible, $V$ and $\tilde{V}$ have the same rank.

Define a linear transform $T: V\rightarrow \tilde{V}$ over spanning vectors:
\begin{align*}
&T(R_n) = \tilde{R}_{n+1},\\
&T(R_j) = \tilde{R}_j, \quad \textrm{ for } \; j\neq n.
\end{align*}
Because of the linear relation $\mu_{\{n\}}^{-1}\tilde{R}_{n} + \tilde{R}_{n+1} = 0$ in $\tilde{V}$, the linear map $T$ is a surjective, and it is an isomorphism of vector spaces. In terms of spanning vectors, $T$ can be written as a $(n+1)\times n$ matrix:
$$
T = 
\begin{pmatrix}
I_{n-1} & 0 \\
0 & 0\\
0 & 1
\end{pmatrix}.
$$

Finally we prove that $T$ is an isomorphism of representations, namely for $1\leq t\leq n$,
\begin{equation*}\label{PNsisorep}
\tilde{\rho}(m_t)\circ T = T\circ \rho(m_t).
\end{equation*}
Let $M_t =\rho(m_t)$ and $\tilde{M}_t = \tilde\rho(m_t)$, there is 
$$M_t = I_n - \sum_{j=1}^n\epsilon(\gamma_{tj})E_{tj}, \quad \tilde{M}_t = I_{n+1} - \sum_{j=1}^{n+1}\epsilon(\tilde{\gamma}_{tj})E_{tj}.$$

If $t\neq n$, then $m_t= \tilde{m_t}$, $\gamma_{tj} = \tilde{\gamma}_{tj}$. It is equivalent to check $(I_{n+1} - \tilde{M}_t)\circ T = T\circ (I_n - M_t)$.  By a straightforward calculation, both hand sides equal to an $(n+1)\times n$ matrix where the $(t,j)$ entry equals to $\epsilon(\gamma_{tj})$ and zero otherwise.

If $t= n$, then we take $m_t = \tilde{m}_{n+1}$. It remains to check $(I_{n+1} - \tilde{M}_{n+1})\circ T = T\circ (I_n - M_n)$. Both hand sides equal to an $(n+1)\times n$ matrix where the $(n+1,j)$ entry equals to $\epsilon(\gamma_{nj})$ and zero otherwise.

The group actions of the meridian generators are compatible on both $V$ and $\tilde{V}$, so is the entire link group $\pi_L$. Therefore, the augmentation representations before and after a positive stabilization are isomorphic. 

\smallskip
The proof for a negative stabilization is similar, which we do not repeat here. For reference, we record some identities and the matrix $\tilde{R}$ after a negative stabilization. For $1\leq i, j \leq n-1$, the framed cords are identified by the following relations
\begin{align}
\tilde\gamma_{ij} &= \gamma_{ij}, \nonumber\\
\tilde{\gamma}_{i,n} &= \tilde\gamma_{i,n+1} = \gamma_{in} \cdot (m_{n}^{(n)})^{-1},\nonumber\\
\tilde\gamma_{n,j} &= \tilde\gamma_{n+1, j} = (m_n^{(n)}) \cdot \gamma_{nj}, \nonumber \\
\tilde{\gamma}_{nn} &= \tilde\gamma_{n,n+1} = \tilde\gamma_{n+1,n}  = \tilde{\gamma}_{n+1,n+1} =\gamma_{nn}, \nonumber
\end{align}
the meridian generators are identified in the same way as in (\ref{meridiansPoSt}), and the matrix $\tilde{R}$ is

\begin{equation}
\tilde{R} = 
\begin{pmatrix}
\epsilon(\gamma_{11}) & \dotsb & \epsilon(\gamma_{1,n-1}) & \mu_{\{n\}}^{-1}\epsilon(\gamma_{1n}) & \mu_{\{n\}}^{-1}\epsilon(\gamma_{1n}) \\
\vdots & \ddots & \vdots & \vdots &\vdots \\
\epsilon(\gamma_{n-1,1}) & \dotsb &\epsilon(\gamma_{n-1,n-1}) & \mu_{\{n\}}^{-1}\epsilon(\gamma_{n-1,n}) &  \mu_{\{n\}}^{-1}\epsilon(\gamma_{n-1,n})\\
\mu_{\{n\}}\epsilon(\gamma_{n1}) & \dotsb & \mu_{\{n\}}\epsilon(\gamma_{n, n-1}) & \epsilon(\gamma_{nn}) &\epsilon(\gamma_{nn})\\
\mu_{\{n\}}\epsilon(\gamma_{n1}) & \dotsb & \mu_{\{n\}}\epsilon(\gamma_{n, n-1}) & \epsilon(\gamma_{nn}) & \epsilon(\gamma_{nn})
\end{pmatrix}.
\end{equation}

\smallskip
So far we have proven that the augmentation representation constructed from a braid is invariant under either a positive or a negative stabilization.

\medskip
\textbf{Summary.}
In previous arguments, we showed that the augmentation representation constructed from a braid is invariant under a conjugation or a positive/negative stabilization. Together with the case of braid equivalences, we can apply Markov's theorem and prove that the augmentation representation is well-defined up to isomorphism. In particular, it is independent from the choice of the braid representative of the link. 
\end{proof}

\begin{rmk}
From the perspective of knot contact homology, a based loop concatenated by capping paths models the boundary of a holomorphic disk with one puncture asymptotic to the Reeb chord (or the framed cord), motivating the construction in the theorem.
\end{rmk}

\begin{rmk}
In this remark, we compare the construction with those in \cite{Co, Ga2}, explaining what has been generalized.

When $K$ is a knot, we can either write it as a braid closure and apply Theorem-Definition \ref{MainConstruction} to construct the augmentation representation for $\epsilon: \mathrm{Cord}(K)\rightarrow k$, or we can choose a set of meridian generators and apply \cite{Co} or \cite[Propsition 4.11]{Ga2} to construct a representation for $\epsilon: \cP_K\rightarrow k$.

Both constructions require additional data, and in both cases, the representation does not depend on additional data. The arguments for independence are different. In this paper, we prove that Markov moves give isomorphic representations. In \cite{Co} or \cite{Ga2}, the independence of the generating set of meridians was proven either by the irreducibility of the representation, or by a classification of   simple sheaves microsupported along the knot conormal.

Now we explain how the current paper manifests the underlying geometry of the construction in \cite{Co, Ga2}. It is not an immediate identification of the two constructions using the $\bbZ$-algebra isomorphism $\textrm{Cord}^c(K) \cong \cP_K$ in Remark \ref{3fcas}. Thanks to the fact that the representation does not depend on the additional data in either constructions, we can make a preferred choice for the position of the knot as well as the set of meridian generators.

Suppose $K$ is the braid closure of an $n$-strand braid $B$ and let $\{m_t\}_{1\leq t\leq n}$ be the set of meridian generators in the configuration disk $D$ associated to $B$. We choose $x_1\in D$ to be the base point of $\pi_L$, then our preferred set of meridian generators is
$$\{m_t^{(1)}\}_{1\leq t\leq n}.$$
Since any two meridians are conjugate to each other in the knot group, there exists $g_t\in \pi_L$ for each $t$, such that $m_t^{(1)} = g_t^{-1}\cdot m_1^{(1)}\cdot g_t$.
\end{rmk}

Both constructions starts with a square matrix $R$ of size $n\times n$,
$$(R_{\textrm{Cord}(K)})_{ij} = \epsilon_{\textrm{Cord}(K)}(\gamma_{ij}), \qquad (R_{\cP_K})_{ij} = \epsilon_{\cP_K}(g_i\cdot g_j^{-1}).$$

\noindent We present an entry-wise identification between the two matrices. Recall that $\ast \in \ell$ is the marked point used in the longitude relation in the cord algebra. For each $t\in \{1,\dotsb, n\}$, there is a unique path $d_{t1}$ from $x_t$ to $x_1$ that is contained in $\ell \setminus \{\ast\}$ along the natural orientation of $\ell$. There are
$$m_t^{(1)} = \gamma_{1t} \cdot m_t^{(t)}\cdot \gamma_{t1},\quad m_t^{(t)} = d_{t1} \cdot m_1^{(1)}\cdot d_{t1}^{-1}.$$
Combining these two equations with $m_t^{(1)} = g_t^{-1}\cdot m_1^{(1)}\cdot g_t$, we get
\begin{equation}\label{gtchoice}
g_t = (\gamma_{1t}\cdot d_{t1})^{-1}.
\end{equation}
Note that the $g_t$ here is a choice we made which is not unique in general. For example, the longitude commutes with the meridian in the peripheral subgroup. The concatenation of a choice $g_t$ with the longitude gives another choice. Fix the chosen $g_t$ in (\ref{gtchoice}), we have
$$g_i\cdot g_j^{-1} = (\gamma_{1i}\cdot d_{i1})^{-1}\cdot (\gamma_{1j}\cdot d_{j1}) = d_{i1}^{-1}\cdot \gamma_{ij}\cdot d_{j1}.$$
Note that if $c_{ij}$ is a framed cord, then $d_{i1}^{-1}\cdot c_{ij}\cdot d_{j1}$ is a based loop in $\pi_L$. Moreover, $[c_{ij}]\mapsto [d_{i1}^{-1}\cdot c_{ij}\cdot d_{j1}]$ defines an isomorphism $\textrm{Cord}(K) \xrightarrow {\sim} \cP_K$. In particular, there is
$$\epsilon_{\cP_K}(g_i\cdot g_j^{-1}) = \epsilon_{\textrm{Cord}(K)}(\gamma_{ij}),$$
proving $R_{\textrm{Cord}(K)} = R_{\cP_K}$. It is similar to check that the group actions also match. 

\section{A microlocal digression}\label{Sec:micro}

We defined the augmentation representation in the previous section. It is a representation of the fundamental group of the link complement, which is equivalent to a local system on the link complement. We will study properties of the augmentation representation in the next section. It is easier to understand these properties from the perspective of microlocal sheaf theory. Indeed, a local system can be viewed as a locally constant sheaf, and we can study it with microlocal methods. In this section, we review some microlocal sheaf theory.

Let $k$ be a commutative field. Let $Y$ be a smooth manifold. Let ${Mod}(Y)$ be the abelian category of sheaves of $k$-modules on $Y$, and $Sh(Y)$ be the bounded dg derived category. For any object $\cF\in Sh(Y)$, its microsupport $SS(\cF)\subset T^*Y$ is a closed conic subset, point-wisely defined as follows (also see \cite[Definition 5.1.1]{KS}).

\begin{defn}
Let $\cF\in Sh(Y)$ and let $p = (y_0,\xi_0)\in T^*Y$. We say that $ p \notin {SS}(\cF)$ if there exists an open neighborhood $U$ of $p$ such that for any $y\in Y$ and any real $C^1$ function $\phi$ on $Y$ satisfying $d\phi(y_0) \in U$ and $\phi(y_0)=0$, we have
\begin{equation}\label{microsupport}
R\Gamma_{\{\phi(y)\geq 0\}}(\cF)_{y_0}\cong 0.
\end{equation}
\end{defn}

Kashiwara-Schapira defined the microlocal hom bifunctor \cite[Definition 4.1.1]{KS}
\begin{equation*}\label{muhom}
\mu hom: Sh(Y)^{op}\otimes Sh(Y)\rightarrow Sh(T^*Y),
\end{equation*}
giving a quantitative description of the micro-support. Suppose $\cF,\cG \in Sh(Y)$, then 
\begin{equation*}\label{muhomSSbound}
\textrm{supp } \mu hom (\cF,\cG) \subset {SS}(\cF)\cap {SS}(\cG).
\end{equation*}

Let $T^\infty Y$ be the unit cosphere bundle of $Y$. It admits a natural contact form induced from the canonical form of the cotangent bundle. Let $\Lambda\subset T^\infty Y$ be a (not necessarily connected) smooth Legendrian submanifold. Define a dg subcategory
$$Sh_\Lambda(Y) := \{\cF \in Sh(Y) \,|\, SS(\cF)\cap T^\infty Y \subset \Lambda\}.$$
Following \cite{GKS}, $Sh_\Lambda(Y)$ is a Legendrian isotopy invariant. The abelian category $Mod(Y)$ can be viewed as a dg subcategory of $Sh(Y)$ consisting of objects which are concentrated in homological degree zero. Define $Mod_\Lambda(Y) := Sh_\Lambda(Y)\cap Mod(Y)$.

A sheaf $\cF \in Sh_\Lambda(Y)$ is simple along $\Lambda$ if one of the following equivalent conditions holds:
\begin{itemize}
\item
For any $p\in \Lambda$, the microlocal Morse cone appeared in (\ref{microsupport}) has rank $1$, namely
$$R\Gamma_{\{\phi(y)\geq 0\}}(\cF)_{y_0} \cong k[d],\quad \textrm{for some }\ d\in \bbZ.$$
\item 
The self microlocal hom restricts to a constant sheaf supported on $\Lambda$, namely
$$\mu hom(\cF,\cF)|_{T^\infty Y} = k_{\Lambda}.$$
\end{itemize}
We write $Sh^s_\Lambda(Y)\subset Sh_\Lambda(Y)$ for the subcategory of simple sheaves along $\Lambda$. The triangulated structure is lost when we pass to this subcategory.

There is a distinction between a sheaf $\cF \in Sh_\Lambda(Y)$ being ``simple along $\Lambda$'' or ``simple along its micro-support''. These two notions are equivalent when $SS(\cF)\cap T^\infty Y = \Lambda$. In the definition $Sh_\Lambda(Y)$, an object $\cF$ is only required to have its micro-support intersecting $T^\infty Y$ in a subset of $\Lambda$. When $\Lambda$ has multiple connected components, such as in our case the conormal tori of links, the first notion is strictly stronger (defining fewer objects).

\smallskip
In this paper, we consider $Y = X = \bbR^3$ or $S^3$, and $\Lambda = \Lambda_L$, the Legendrian conormal tori of a link $L$. Simple sheaves microsupported along $\Lambda$ admit easier descriptions. The description for links is similar to that for knots, which was studied in \cite{Ga2}, because both the micro-support and the simpleness are local properties.

For simplicity, we consider sheaves concentrated in homological degree $0$. The micro-support constraints force that $\cF$ restricted to each component of the link, or the link complement is a local system \cite[Lemma 3.1]{Ga2}. In terms of group representations, these local systems are equivalent to
\begin{align*}
&\rho: \pi_L \rightarrow GL(V), \\
&\rho_i: \bbZ_{K_i} \rightarrow GL(W_i),  \;\;\textrm{for}\;\;  1\leq i\leq r.
\end{align*}
Here $\pi_L$ is the link group and $\bbZ_{K_i} := \pi_1(K_i)$. Conversely, we can reconstruct the sheaf from these local systems by gluing. The gluing data is an element in an extension class, and according to \cite[Lemma 3.2, Lemma 3.3]{Ga2}, it is equivalent to a collection of linear maps
$$T_i: W_i\rightarrow V,   \;\;\textrm{for}\;\; 1\leq i\leq r,$$
which satisfy the following compatibility conditions. For each $1\leq i\leq r$, let $m_i, \ell_i$ be the meridian and longitude in the peripheral subgroup, then $T_i$ satisfies (a) $\rho(\ell_i)\circ T_i = T_i\circ \rho_i(K_i)$, and (b) $m_i$ acts on the image of $T_i$ as identity. We conclude that the sheaf $\cF$ is equivalent to the collection of data $(\rho, V, \rho_i, W_i, T_i)$.

A sheaf $\cF$ is simple along $\Lambda$ if and only if $cone(T_i)$ has rank $1$ for each $1\leq i\leq r$. It is simple along its singular support if and only if $cone(T_i)$ has rank at most $1$ for each $1\leq i\leq r$. These statements follow from the first definition of simpleness.

In this paper, we focus on the representation of the link group, namely $(\rho,V)$. The simpleness imposes strong restrictions to this representation. If $cone(T_i)$ has rank $1$, then $T_i$ is either injective with a rank $1$ cokernel, or $T_i$ is surjective with a rank $1$ kernel. Consider the first case when $T_i$ is injective, then $W_i$ can be regarded as a subspace of $V$ of codimension $1$. The condition (b) of $T_i$ requires that $m_i$ acts on a space of codimension one as identity. In the other cases, including when rank$(cone(T_i)) =1$ with surjective $T_i$, or when rank$(cone(T_i)) = 0$, the action on $m_i$ is entirely trivial. Combining these cases, we see that for $\cF$ to be simple, it is necessary for $(\rho,V)$ to satisfy the condition that the action of each meridian $m_i$ fixes a subspace of codimension at most $1$.

\section{Properties of the augmentation representation}\label{Sec:Properties}

In this section, we study properties of the augmentation representation.

\subsection{Microlocal simpleness}

Let $Mod_{\Lambda_L}(X) = Sh_{\Lambda_L}(X)\cap Mod(X)$, and let $Mod^s_{\Lambda_L}(X) \subset Mod_{\Lambda_L}(X)$ be the full subcategory of simple sheaves. Suppose $\cF \in Mod^s_{\Lambda_L}(X)$, then $j^{-1}\cF$ is a local system on $X\setminus L$. The local system is equivalent to a representation $\rho: \pi_L\rightarrow GL(V)$. The simpleness of $\cF$ requires that any meridian acts on $V$ as identity on a subspace of codimension $1$ or $0$. Based on this observation, we make the following notion of simpleness.

\begin{defn}
Let $V$ be a vector space. A linear automorphism $T \in GL(V)$ is \textit{almost identity} if there is a subspace $W\subset V$ of codimension $1$ such that $T|_W = \mathrm{id}_W$.

Equivalently, $T$ is almost identity if and only if the rank of $(\mathrm{id}_V - T)$ is at most $1$.
\end{defn}

\begin{defn}
Suppose $L \subset X$ is a link. A link group representation
$$\rho: \pi_L \rightarrow GL(V)$$
is \textit{microlocally simple} if $\rho(m)$ is almost identity for every meridian $m\in \pi_L$.
\end{defn}

\begin{prop}\label{propSimpleness}
Augmentation representations are microlocally simple.
\end{prop}
\begin{proof}
Let $(\rho_\epsilon, V_\epsilon)$ be the augmentation representation associated to an augmentation $\epsilon: \mathrm{Cord}(L)\rightarrow k$. Recall that $R$ is the $n\times n$ matrix determined by $\epsilon$. As a vector space, $V_\epsilon$ is spanned by the column vectors $R_j$, $1\leq j \leq n$.  To verify that a linear automorphism $T \in GL(V_\epsilon)$ is almost identity, it is sufficient to show that $(\textrm{id}_V - T)R_j$ is contained in a rank $1$ subspace of $V$ for all $1\leq j \leq n$.

To prove $(\rho_\epsilon, V_\epsilon)$ is microlocally simple, we need to verify that every $T = \rho_\epsilon(m)$ with $m$ being a meridian is almost identity. Since any two meridians with coherent orientations belonging to the same component $K_i\subset L$ are conjugate, it suffices to select a meridian $m_i$ for each component $K_i$ and prove that $\rho_\epsilon(m_i)^{\pm 1}$ are almost identity. It is easy to see that if $\rho_\epsilon(m_i)$ is almost identity, so is $\rho_\epsilon(m_i)^{-1}$. Finally, recall that $\{m_t\}_{1\leq t\leq n}$ are meridian generators of the link group $\pi_L$. We observe that the generating set contains at least one meridian for each component.

Following these arguments, the assertion in the proposition reduces to show that for any generating meridian $m_t$, there exists a rank $1$ subspace $U_t\subset V_\epsilon$ such that 
$$(\textrm{id}_V-\rho_\epsilon(m_t))R_j \subset U_t, \quad\textrm{for } 1\leq j \leq n.$$
Recall the formula in (\ref{MeridianAction}), $\rho_\epsilon(m_t) R_j = R_i - \epsilon({p}_t^{-1}\cdot p_j) R_t$. We take $U_t = \textrm{Span}_k\{R_t\}$, then
$$(\textrm{id}_V -\rho_\epsilon(m_t))R_j = R_j - (R_i - \epsilon({p}_t^{-1}\cdot p_j) R_t) = \epsilon({p}_t^{-1}\cdot p_j) R_t \subset U_t.$$

We prove the desired result.
\end{proof}

\begin{rmk}
Microlocally simple knot group representations are called ``KCH representations'' or ``unipotent KCH representations'' in earlier papers \cite{Ng4,Co, Ga2}. We stop using these names for two reasons. First, the abbreviation ``KCH'' as a prefix of the representation emphasizes more on contact topology and its relation to sheaf theory, instead of the property of the representation itself. It is more revealing to borrow the notion of simpleness from microlocal sheaf theory. Second, it is a locally property whether a meridian action is diagonalizable or unipotent. When we work with links, the previous naming system fails to generalize in a concise way when meridians from different link components act differently. On the other hand, the word ``microlocal'' implies to consider each component separately.
\end{rmk}

\begin{rmk}
We answer an interesting question from Emmanuel Giroux. Consider the rank $1$ trivial representation of the link group. It is microlocally simple, but does not come from an augmentation. Indeed, if the action of any loop is trivial, in particular the action of any meridian generator is trivial, then augmented values of standard cords must all vanish (by \eqref{MeridianAction1}). As a consequence, there must be $V_\epsilon =\{0\}$ and the augmentation representation is a zero representation, not the rank $1$ trivial representation.

To resolve this issue, we need to look into the microlocal sheaf category. Consider the case when $K$ is a knot and we use the main result in \cite{Ga2} as an example. If all standard cords are augmented to zero, then the augmentation defines a sheaf $i_*\cG_{\epsilon(\lambda)}[-1]$ where $i: K\rightarrow X$ is the closed embedding and $\cG_{\epsilon(\lambda)}$ is a rank $1$ local system on $K$ with monodromy $\epsilon(\lambda)$. If $\epsilon(\lambda) =1$, then $\cG_{\epsilon(\lambda)} = k_K$ and there is distinguished triangle in $Sh(X)$:
$$i_*k_K[-1] \rightarrow j_! k_{X\setminus K}\rightarrow k_X\xrightarrow{+1},$$
here $j: X\setminus K\rightarrow X$ is the open embedding. The sheaf $j_! k_{X\setminus K}$ is equivalent to the rank $1$ trivial representation of $\pi_L$. The distinguished triangle yields that $i_*\cG_{\epsilon(\lambda)}[-1]$ and $j_! k_{X\setminus K}$ are isomorphic up to locally constant sheaves on $X$.
\end{rmk}

\subsection{Vanishing}

Suppose $L_0\subset L$ is a sublink. Let $\pi_{L_0} = \pi_1(X\setminus L_0)$. Note this fundamental group is taken over the complement of $L_0$ in $X$, forgetting that $L_0$ is a sublink of $L$. The natural open inclusion $j: X\setminus L\rightarrow X\setminus L_0$ induces a map on their fundamental groups:
$$\pi_L\rightarrow \pi_{L_0}.$$
It induces a functor between abelian categories $Rep(\pi_{L_0})\rightarrow Rep(\pi_L)$. If we identify representations of a fundamental group as local systems, this functor coincides with the pull back functor $j^{-1}: loc(X\setminus L_0)\rightarrow loc(X\setminus L)$.

In the previous subsection, we proved that augmentations are microlocally simple. It follows from the definition that the group action of a meridian has two possibilities. It either fixes the entire vector space, or defines an invariant subspace of codimension one. In this subsection, we give a sufficient condition such that the associated augmentation representations fit into the first case.

\begin{defn}
Let $L\subset X$ be a link and $L_0\subset L$ be a sublink. A simple link representation $\rho: \pi_L\rightarrow GL(V)$ vanishes on $L_0$ if $\rho(m_0) = \textrm{id}_V$ for any meridian $m_0$ of $L_0$.
\end{defn}

The definition comes from an observation in the sheaf theory. Let $\cE\in loc(X\setminus L)$ be the local system determined by a link group representation $(\rho,V)$. Let $j: X\setminus L\rightarrow X$ be the open embedding. We consider the underived push forward $\cF = j_*\cE$. If $(\rho,V)$ vanishes on $L_0$, then $\cF$ is microsupported on the component of $\Lambda_{L_0}$, i.e.
$$SS(\cF) \cap T^\infty X \subset  \Lambda_{L\setminus L_0}.$$

\begin{prop}\label{propVanishing}
Let $L \subset X$ be a link and $L_0 \subset K$ a sublink. Suppose an augmentation $\epsilon: \mathrm{Cord}(L)\rightarrow k$ satisfies either one the following conditions:
\begin{itemize}
\item
for any framed cord $c$ starting from $L_0$, $\epsilon(c) = 0$; or
\item 
for any framed cord $c$ ending on $L_0$, $\epsilon(c) = 0$,
\end{itemize}
then the associated augmentation representation $(\rho_\epsilon, V_\epsilon)$ vanishes on $L_0$.
\end{prop}

\begin{proof}
By Lemma \ref{sortingLem}, we can assume there is an integer $s$ with $1\leq s\leq n$, such that the closure of strands $\{1,\dotsb, s\}$ is precisely the sublink $L_0$.

By definition, the associated augmentation representation vanishes on $L_0$ if for any meridian $m_t$, $1\leq t\leq s$, the action $\rho_\epsilon(m_t)$ is identity. Recall from (\ref{MeridianAction}) that 
$$\rho_\epsilon(m_t) R_j = R_j- \epsilon({p}_t^{-1}\cdot p_j) R_t.$$
It suffices to show that either $\epsilon({p}_t^{-1}\cdot p_j) = 0$ for any $j \in \{1,\dotsb, n\}$, or $R_t$ is the zero vector.

Suppose $\epsilon$ maps all cords starting on $L_0$ to zero. For any $t\in \{1,\dotsb, s\}$ and any $j \in \{1,\dotsb, n\}$, $p_t^{-1}\cdot p_j$ is a cord starting on $L_0$. Hence $\epsilon({p}_t^{-1}\cdot p_j) = 0$ and the assertion follows.

Suppose $\epsilon$ maps all cords ending on $L_0$ to zero. For any $t\in \{1,\dotsb, s\}$, $R_t$ consists of augmentations of cords ending on $L_0$. Therefore $R_t$ is a zero vector as expected.

We prove the assertion in both cases. 
\end{proof}

\subsection{Separability}

In this section, we give a sufficient condition for which the augmentation representation is separable. We begin with an example to motivate this notion.

Consider the two-component unlink and the Hopf link (Figure \ref{Fig:UnlinkandHopf}). Each of them has two unknotted components. A link is \textit{split} if it is the union of two sublinks that lie in two disjoint solid balls. In our example, the unlink is split while the Hopf link is non-split.

\begin{figure}[h]
	\centering
	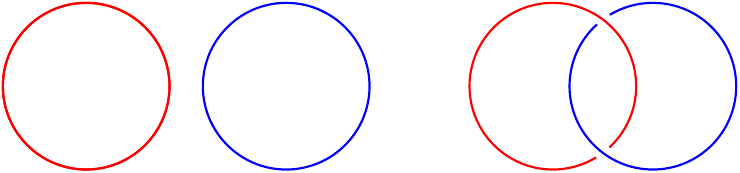
	\caption[]{Two-component unlink (left) and the Hopf link (right).}
	\label{Fig:UnlinkandHopf}
\end{figure}

In either case, the framed cord algebra is generated over $\bbZ[\mu_1^{\pm 1}, \lambda_2^{\pm 1}, \mu_2^{\pm 1},\lambda_2^{\pm 1}]$ by four cords, $a_{11}, a_{12}, a_{21}, a_{22}$. The subscripts label the link components where the end points of a cord belong. (They are not standard cords, which is why we did not use $\gamma_{ij}$.) In $\textrm{Cord}^c(L_{\textrm{Hopf}})$, generators are related by:
\begin{align*}
& (\lambda_1\mu_1\lambda_2^{-1}\mu_2^{-1}-1)a_{12}=0,\\
& a_{21}(1-\lambda_1^{-1}\mu_1^{-1}\lambda_2\mu_2)=0,\\
& 1-\lambda_1-\mu_1 +\lambda_1\mu_1+\lambda_1\mu_1\mu_2^{-1}a_{12}a_{21}=0,\\
& 1 -\lambda_2-\mu_2 +\lambda_2\mu_2 + \lambda_2a_{21}a_{12}=0, \\ 
& a_{12} +\lambda_1\mu_2^{-1}a_{12}(\mu_1\mu_2 - \mu_1 - \mu_2 + \mu_1a_{12}a_{21})=0, \\
& \lambda_2a_{21}-\mu_1a_{21}=0.
\end{align*}
Pure cords do not appear in these equations, because for the unlink and the Hopf link, any pure cord can be homotopic to a constant cord, and then replaced by $1-\mu_{\{i\}}$ using the meridian relation. An augmentation $\epsilon: {\textrm{Cord}^c(L_{\textrm{Hopf}})}\rightarrow k$ is in one the following two non-exclusive cases (we abbreviate $\epsilon(\mu_i),\epsilon(\lambda_i)$ as $\mu_i,\lambda_i$):
\begin{enumerate}
\item[(A)]
$\mu_1 = \lambda_2$, $\mu_2 = \lambda_1$, and $\epsilon(a_{12})\epsilon(a_{21}) = (1-\mu_1^{-1})(1-\mu_2)$; or
\item[(B)]
$\epsilon(a_{12}) = \epsilon(a_{21}) = 0$, and $1 - \lambda_1 - \mu_1 + \lambda_1\mu_1 = 1 - \lambda_2 - \mu_2 +\lambda_2\mu_2 = 0$.
\end{enumerate}
Recall the augmentation variety in \cite{Ng4},
$$V_K =\big\{\big(\epsilon(\mu_1), \epsilon(\lambda_1),\dotsb, \epsilon(\mu_r), \epsilon(\lambda_r)\big)\,|\, \epsilon: {\textrm{Cord}^c(K_{\textrm{Hopf}})}\rightarrow k\big\}\subset (k^*)^{2r}.$$
Therefore the augmentation variety for the Hopf link is
$$V_{\textrm{Hopf}} 
=\left \{
\begin{matrix}
\mu_1 = \lambda_2 \\
\mu_2 = \lambda_1
\end{matrix}
\right \}
\cup
\left \{
\begin{matrix}
1 - \lambda_1 - \mu_1 + \lambda_1\mu_1 =0\\
1 - \lambda_2 - \mu_2 +\lambda_2\mu_2 =0
\end{matrix}
\right\} = V_A \cup V_B.
$$
We can similarly compute the augmentation variety for the two-component unlink
$$V_{\textrm{unlink}} 
=
\left \{
\begin{matrix}
1 - \lambda_1 - \mu_1 + \lambda_1\mu_1 =0\\
1 - \lambda_2 - \mu_2 +\lambda_2\mu_2 =0
\end{matrix}
\right\} =  V_B.
$$

We observe that the augmentation variety of the unlink is contained in that of Hopf link, corresponding to case (B) when $\epsilon(a_{12}) = \epsilon(a_{21}) =0$. In this case, all mixed cords are augmented to zero. Though the Hopf link is non-split, some augmentations behave as if the framed cord algebra came from a split link. We remark that the idea of sending mixed cords to $0$ has been considered in contact geometry such as \cite{Mi, AENV}.

\smallskip
We propose a counterpart of this phenomenon on the sheaf side. The following lemma shows that the direct sum of two simple sheaves which are microsupported along disjoint Legendrians is again simple. 

\begin{lem}\label{DirectSum}
Let $Y$ be a manifold, and $\Lambda_1,\Lambda_2\subset T^\infty Y$ two disjoint Legendrian submanifolds. If $\cF_1 \in Sh^{s}_{\Lambda_1}(Y), \cF_2 \in Sh^{s}_{\Lambda_2}(Y)$, then $\cF_1\oplus \cF_2 \in Sh^{s}_{\Lambda_1\sqcup\Lambda_2}(Y)$.
\end{lem}
\begin{proof}
Because $\Lambda_1,\Lambda_2$ are disjoint, $\mu hom(\cF_1,\cF_2)|_{T^\infty Y} = \mu hom(\cF_2,\cF_1)|_{T^\infty Y} = 0$. The simpleness of $\cF_i$ yields $\mu hom(\cF_i,\cF_i)|_{T^{\infty}Y} = k_{\Lambda_i}$. Finally we have
\begin{align*}
\mu hom(\cF_1\oplus \cF_2, \cF_1\oplus \cF_2)|_{T^\infty Y} 
&= \mu hom(\cF_1, \cF_1)|_{T^\infty Y} \oplus \mu hom(\cF_2, \cF_2)|_{T^\infty Y} \\
&= k_{\Lambda_1}\oplus k_{\Lambda_2} = k_{\Lambda_1\sqcup \Lambda_2}.
\end{align*}
\end{proof}

We give a sufficient condition when an augmentation representation splits into direct summands.
\begin{defn}
Suppose $L = L_{1} \sqcup L_{2}$ is the union of two sublinks $L_1, L_2$. A link group representation $\rho: \pi_L\rightarrow GL(V)$ is separable with respect to the partition if there exist link group representations $\rho_i: \pi_{K_i} \rightarrow GL(V_i)$ for $i= 1,2$, such that
$$(\rho, V) = (\rho_1,V_1)\oplus (\rho_2, V_2),$$
where $V_i$ is considered as a $\pi_L$-representation through the composition 
$$\pi_L\rightarrow \pi_{L_i}\rightarrow GL(V_i).$$
\end{defn}

\begin{prop}\label{separability}
Suppose $L = L_{1} \sqcup L_{2}$ is the union of two sublinks $L_1, L_2$. If an augmentation $\epsilon: \mathrm{Cord}(L)\rightarrow k$ maps all mixed cord between $L_1$ and $L_2$ to zero, then the induced augmentation representation is separable with respect to the partition.
\end{prop}

\begin{proof}
Generally speaking, if $L_1\subset L$ is a sublink, there is no natural morphism $\mathrm{Cord}(L_1)\rightarrow \mathrm{Cord}(L)$. Instead, the natural morphism exists in the opposite direction.
$$\mathrm{Cord}(L)\rightarrow \mathrm{Cord}(L_1).$$

When all mixed cords are augmented to zero, an augmentation $\epsilon: \mathrm{Cord}(L)\rightarrow k$ factors through this morphism, giving an induced augmentation $\epsilon_1: \mathrm{Cord}(L_1)\rightarrow k$. A generic cord in $\mathrm{Cord}(L_1)$ is a pure cord in $\mathrm{Cord}(L)$, except when the cord intersects $L_2$ transversely. In this singular case, the cord can be perturbed away from $L_2$ in two ways, say $c,c'$, and they differ by an interpolation of a meridian $m$ -- there exist two mixed cords $c_{12}$ and $c_{21}$ such that $c = c_{12}\cdot c_{21}$ and $c' = c_{12}\cdot m \cdot c_{21}$. By the skein relation and $\epsilon(c_{12}) = \epsilon(c_{21})=0$, we have
$$\epsilon(c) - \epsilon (c') = \epsilon(c_{12})\epsilon(c_{12})=0.$$
Hence $\epsilon(c)$ is well-defined if we regard $[c]$ as a cord class in $\mathrm{Cord}(L_1)$. Relations in $\mathrm{Cord}(L_1)$ also hold in $\mathrm{Cord}(L)$. Hence the induced augmentation $\epsilon_1: \mathrm{Cord}(L_1)\rightarrow k$ is well-defined.

By Theorem \ref{MainConstruction}, the augmentation $\epsilon_i$ induces an augmentation representation $(\rho_i, V_i)$, $i=1,2$. By the natural morphism $\pi_L\rightarrow \pi_{L_i}$, $V_1\oplus V_2$ is a $\pi_L$-representation.

By Lemma \ref{sortingLem}, we assume there is an integer $s$ such that the closure of strands $P_1 := \{1,\dotsb, s\}$ is $L_1$, and the closure of strands $P_2 :=\{s+1,\dotsb, n\}$ is $L_2$.

Since all mixed cords are augmented to zero, the matrix $R$ reduces to
$$
\begin{pmatrix}
\tilde{R}_1  & 0 \\
0 & \tilde{R}_2
\end{pmatrix}
$$
where $\tilde{R}_1$ has size $s \times s$ and $\tilde{R}_2$ has size $(n-s)\times (n-s)$. 
We can take the configuration disk to construct $(\rho_1, V_1)$ to be a subset of the configuration disk to construct $(\rho, V)$, which yields $V_1 = \textrm{Span}_k\{(\tilde{R}_1)_j\}_{1\leq j\leq s} = \textrm{Span}_k\{{R}_j\}_{1\leq j\leq s}$. Similar for $V_2$. Therefore, $V_\epsilon = V_1\oplus V_2$ as vector spaces. It remains to show that the isomorphism respects $\pi_L$-actions.

Let $m$ be any meridian of $L_1$. Then $\rho_2(m)=0$ because $m$ is contractible in $\pi_{L_2}$, and $\rho_1(m)$ acts on $(\tilde{R}_1)_j$by interpolating $m$ in pure standard cords. For $\rho_\epsilon(m)$, its action on $R_j$, $1\leq j \leq s$ is the same interpolation, and the action on $R_j$, $s+1\leq j\leq n$ is identity because interpolating a meridian of $L_1$ into a pure cord for $L_2$ does not change the augmented value -- same argument as before. Therefore, we conclude $V_\epsilon = V_1\oplus V_2$ as $\pi_L$-representations.

\end{proof}

\begin{rmk}
An augmentation representation may \emph{not} be irreducible. Any example in Proposition \ref{separability} such that $V_1,V_2$ are non-trivial is reducible.
\end{rmk}

\end{document}